\documentclass[12pt]{amsart}
\usepackage[top=30truemm,bottom=30truemm,left=25truemm,right=25truemm]{geometry}
\usepackage{txfonts}
\usepackage{mathrsfs}

\usepackage{bm}
\usepackage{amsfonts,amssymb}
\usepackage{dsfont}
\usepackage{extarrows}
\usepackage{amsmath}
\usepackage{mathrsfs}
\usepackage{enumerate}
\usepackage{amscd}
\usepackage[all]{xy}
\usepackage[hyperfootnotes=true]{hyperref}
\theoremstyle{plain} 
\newtheorem{theorem}{\indent\bf Theorem}[section]

\theoremstyle{definition} 

\newtheorem{problem}[theorem]{\indent\bf Problem}

\newtheorem{thm}{Theorem}[section]

\newtheorem{lem}[thm]{Lemma}
\newtheorem{prop}[thm]{Proposition}

\theoremstyle{definition}
\newtheorem{defn}{Definition}[section]

\theoremstyle{remark}
\newtheorem{rem}{Remark}[section]

\newcommand{\be}{\begin{equation}}
\newcommand{\ee}{\end{equation}}
\newcommand{\bea}{\begin{eqnarray}}
\newcommand{\eea}{\end{eqnarray}}
\newcommand{\ben}{\begin{eqnarray*}}
	\newcommand{\een}{\end{eqnarray*}}
\newcommand{\bt}{\begin{split}}
	\newcommand{\et}{\end{split}}
\newcommand{\bet}{\begin{equation}}
\newcommand{\mc}{\mathbb{C}}

\newcommand{\ra}{\rightarrow}
\newcommand{\dbar}{\bar{\partial}}

\begin{document}
\baselineskip 19pt
\baselineskip 16pt

\title{Positivity of holomorphic vector bundles in terms of  $L^p$-conditions of $\bar\partial$}

\author[F. Deng]{Fusheng Deng}
\address{Fusheng Deng: \ School of Mathematical Sciences, University of Chinese Academy of Sciences\\ Beijing 100049, P. R. China}
\email{fshdeng@ucas.ac.cn}
\author[J. Ning]{Jiafu Ning\textsuperscript{*}}
\address{Jiafu Ning: \ Department of Mathematics, Central South University, Changsha, Hunan 410083, P. R. China.}
\email{jfning@csu.edu.cn}
\author[Z. Wang]{Zhiwei Wang\textsuperscript{*}}
\address{ Zhiwei Wang: \ School
	of Mathematical Sciences\\Beijing Normal University\\Beijing\\ 100875\\ P. R. China}
\email{zhiwei@bnu.edu.cn}
\author[X. Zhou]{Xiangyu Zhou}
\address{Xiangyu Zhou: Institute of Mathematics\\Academy of Mathematics and Systems Sciences\\and Hua Loo-Keng Key
	Laboratory of Mathematics\\Chinese Academy of
	Sciences\\Beijing\\100190\\P. R. China}
\address{School of
	Mathematical Sciences, University of Chinese Academy of Sciences,
	Beijing 100049, China}
\email{xyzhou@math.ac.cn}

\begin{abstract}
We study the positivity properties of Hermitian (or even Finsler) holomorphic vector bundles
in terms of $L^p$-estimates of $\bar\partial$ and $L^p$-extensions of holomorphic objects.
To this end,  we introduce four conditions, called the optimal $L^p$-estimate condition,
the multiple coarse $L^p$-estimate condition, the optimal $L^p$-extension condition,
and the multiple coarse $L^p$-extension condition, for a  Hermitian (or Finsler) vector bundle $(E,h)$.
The main result of the present paper is to give a characterization of the Nakano positivity of $(E,h)$
via the optimal $L^2$-estimate condition.
We also show that  $(E,h)$ is Griffiths positive if it satisfies the multiple coarse $L^p$-estimate condition for some $p>1$,
the optimal $L^p$-extension condition, or the multiple coarse $L^p$-extension condition for some $p>0$.
These results can be roughly viewed as converses of H\"{o}rmander's $L^2$-estimate of $\bar\partial$
 and Ohsawa-Takegoshi type extension theorems.
As an application of the main result, we get a totally different method to Nakano positivity of direct image sheaves of
twisted relative canonical bundles associated to holomorphic families of complex manifolds.

\end{abstract}

\thanks{(*) The second author and the third author are both corresponding authors.}

\maketitle


\tableofcontents

\section{Introduction}


 The present paper is to study positivity properties of Hermitian (or even Finsler) holomorphic vector bundles
 via $L^p$-estimates of $\bar\partial$ and $L^p$-extensions of holomorphic objects,
 which can be roughly viewed as converses of H\"{o}rmander's $L^2$-estimate of $\bar\partial$
 and Ohsawa-Takegoshi type extension theorems.
 This is a continuation of the previous work \cite{DNW1} on characterizations of plurisubharmonic
 functions.

To state the main results, we first introduce some notions.

%
%
\begin{defn}\label{def:Lp estimate}
Let $(X,\omega)$ be a K\"{a}hler manifold of dimension $n$, which admits a positive Hermitian holomorphic line bundle,
$(E,h)$ be a (singular) Hermitian vector bundle (maybe of infinite rank) over $X$, and $p>0$.
\begin{itemize}
\item[(1)]
$(E,h)$ satisfies \emph{the optimal $L^p$-estimate condition}
if for any positive Hermitian holomorphic line bundle $(A,h_A)$ on $X$,
for any $f\in\mathcal{C}^\infty_c(X,\wedge^{n,1}T^*_X\otimes E\otimes A)$ with $\bar\partial f=0$,
there is $u\in L^p(X,\wedge^{n,0}T_X^*\otimes E\otimes A)$, satisfying $\bar\partial u=f$ and
$$\int_X|u|^p_{h\otimes h_A}dV_\omega\leq \int_X\langle B_{A,h_A}^{-1}f,f\rangle^{\frac{p}{2}} dV_\omega,$$
provided that the right hand side is finite,
where $B_{A,h_A}=[i\Theta_{A,h_A}\otimes Id_E,\Lambda_\omega]$.

\item[(2)]
$(E,h)$ satisfies \emph{the multiple coarse $L^p$-estimate condition}
if for any $m\geq 1$, for any positive Hermitian holomorphic line bundle $(A,h_A)$ on $X$,
and for any $f\in\mathcal{C}^\infty_c(X,\wedge^{n,1}T^*_X\otimes E^{\otimes m}\otimes A)$ with $\bar\partial f=0$,
there is $u\in L^p(X,\wedge^{n,0}T_X^*\otimes E^{\otimes m}\otimes A)$, satisfying $\bar\partial u=f$ and
$$\int_X|u|^p_{h^{\otimes m}\otimes h_A}dV_\omega\leq C_m\int_X\langle B_{A,h_A}^{-1}f,f\rangle^{\frac{p}{2}} dV_\omega,$$
provided that the right hand side is finite,
where $C_m$ are constants satisfying the
growth condition $\frac{1}{m}\log C_m\ra 0$ as $m\ra\infty$.
\end{itemize}
\end{defn}

\begin{defn}\label{def:Lp extension}
Let $(E,h)$ be a Hermitian  holomorphic vector bundle (maybe of infinite rank) over a domain $D\subset\mc^n$ with a singular Finsler metric $h$,
and $p>0$.
\begin{itemize}
\item[(1)]
$(E,h)$ satisfies \emph{the optimal $L^p$-extension condition}
if for any $z\in D$, and $a\in E_{z}$ with $|a|=1$,
and any holomorphic cylinder $P$ with $z+P\subset D$,
there is $f\in H^0(z+P, E)$ such that $f(z)=a$ and
$$\frac{1}{\mu(P)}\int_{z+P}|f|^p\leq 1,$$
where $\mu(P)$ is the volume of $P$ with respect to the Lebesgue measure.
(Here by a holomorphic cylinder we mean a domain of the form $A(P_{r,s})$
for some $A\in U(n)$ and $r,s>0$,
with $P_{r,s}=\{(z_1,z_2,\cdots,z_n):|z_1|^2<r^2,|z_2|^2+\cdots+|z_n|^2<s^2\}$).

\item[(2)]
$(E,h)$ satisfies \emph{the multiple coarse $L^p$-extension condition} if
for any $z\in D$, and $a\in E_{z}$ with $|a|=1$, and any $m\geq 1$,
there is $f_m\in H^0(D, E^{\otimes m})$ such that $f_m(z)=a^{\otimes m}$ and satisfies the following estimate:
$$\int_D|f_m|^{p}\leq C_m,$$
where $C_m$ are constants independent of $z$ and satisfying the
growth condition $\frac{1}{m}\log C_m\ra 0$ as $m\ra\infty$.
\end{itemize}
\end{defn}

(See \S \ref{subsec:finsler metric} for the definition of singular Finsler metrics.)

\begin{rem}
Similarly, one can define the optimal (resp. multiple coarse) $L^p$-extension condition
for a Hermitian vector bundle $(E,h)$ over a K\"{a}hler manifold $X$.
But it is clear that if $(E,h)$ satisfies the optimal (resp. multiple coarse) $L^p$-extension condition on $X$,
then it admits the same condition when restricted on any open set $D$ of $X$.
So we just focus on bounded domains in Definition \ref{def:Lp extension}.
However, it is not the case for the optimal (resp. multiple coarse) $L^p$-estimate condition
since a positive Hemitian line bundle over an open domain in $X$ may not extend to $X$.
\end{rem}



The conditions defined in Definition \ref{def:Lp estimate}, \ref{def:Lp extension}
for trivial line bundles were studied in \cite{DNW1}.
The multiple coarse $L^p$-extension condition for vector bundles with singular Finsler metrics was introduced in \cite{DWZZ1},
and the multiple coarse $L^2$-estimate condition for Hermitian vector bundles was introduced in \cite{HI},
which was named as the twisted H\" ormander condition there.
A property (called "minimal extension property") that is related to the optimal $L^2$-extension condition was introduced in \cite{HPS16}.

The first and the main result of this paper is the following characterization of Nakano positivity in terms of optimal $L^2$-estimate condition.

\begin{thm}\label{thm:theta-nakano text_intr}
Let $(X,\omega)$ be a  K\"{a}hler manifold of dimension $n$ with a K\" ahler metric $\omega$, which admits a positive Hermitian holomorphic  line bundle,
$(E,h)$ be a smooth Hermitian vector bundle over $X$,
and $\theta\in C^0(X,\Lambda^{1,1}T^*_X\otimes End(E))$ such that $\theta^*=\theta$.
If for any $f\in\mathcal{C}^\infty_c(X,\wedge^{n,1}T^*_X\otimes E\otimes A)$ with $\bar\partial f=0$,
and any positive Hermitian line bundle $(A,h_A)$ on $X$ with $i\Theta_{A,h_A}\otimes Id_E+\theta>0$ on $\text{supp}f$,
there is $u\in L^2(X,\wedge^{n,0}T_X^*\otimes E\otimes A)$, satisfying $\bar\partial u=f$ and
$$\int_X|u|^2_{h\otimes h_A}dV_\omega\leq \int_X\langle B_{h_A,\theta}^{-1}f,f\rangle_{h\otimes h_A} dV_\omega,$$
provided that the right hand side is finite,
where $B_{h_A,\theta}=[i\Theta_{A,h_A}\otimes Id_E+\theta,\Lambda_\omega]$,
then $i\Theta_{E,h}\geq\theta$ in the sense of Nakano.
On the other hand, if in addition $X$ is assumed to have a complete K\"ahler metric,
the above condition is also necessary for that $i\Theta_{E,h}\geq\theta$ in the sense of Nakano.
In particular, if $(E,h)$ satisfies the optimal $L^2$-estimate condition, then $(E,h)$ is Nakano semi-positive.
\end{thm}

We prove Theorem \ref{thm:theta-nakano text_intr} by connecting $\Theta_{E,h}$
with the optimal $L^2$-estimate condition through the Bochner-Kodaira-Nakano identity,
and then using a localization technique to produce a contradiction if
$i\Theta_{E,h}\geq\theta$ is assumed to be not true.



\begin{problem}
 Does Theorem \ref{thm:theta-nakano text_intr} still hold
 if the  optimal $L^2$-estimate condition is replaced by the  optimal $L^p$-estimate condition for some $p\neq 2$?
\end{problem}

\begin{problem}
 Establish results analogous to Theorem \ref{thm:theta-nakano text_intr}
 for holomorphic vector bundles with singular Hermitian metrics.
\end{problem}

\begin{thm}\label{thm:coarse estimate text-intr}
Let $(X,\omega)$ be a K\"{a}hler manifold, which admits a positive Hermitian holomorphic  line bundle, and $(E,h)$ be a   holomorphic vector bundle over $X$ with a continuous Hermitian metric $h$.
If $(E,h)$ satisfies the multiple coarse $L^p$-estimate condition for some $p>1$,  then $(E,h)$ is Griffiths semi-positive.
\end{thm}

\begin{rem}\label{rem:reduce to trivial bundle}
If $X$ admits a strictly plurisubharmonic function,
it is obviously from the proof that, in Theorem \ref{thm:theta-nakano text_intr}
and Theorem \ref{thm:coarse estimate text-intr},
we can take $A$ to be the trivial bundle (with nontrivial metrics).
\end{rem}

The case that $p=2$ and $h$ is H\"older continuous for Theorem \ref{thm:coarse estimate text-intr} was proved
in \cite{HI}, by showing that the multiple coarse $L^2$-estimate condition implies the multiple coarse $L^2$-extension condition
and then applying  \cite[Theorem 1.2]{DWZZ1}.
The case that $E$ is a trivial line bundle was proved in \cite{DNW1}.
Theorem \ref{thm:coarse estimate text-intr} is proved by modifying the technique in \cite{HI, DNW1}.

\begin{thm}\label{thm: optimal Lp extension : Griffiths positive-intr}
Let $E$ be a holomorphic vector bundle over a domain $D\subset\mc^n$,
and $h$ be a singular Finsler metric on $E$,
 such that $|s|_{h^*}$ is upper semi-continuous
for any local holomorphic section $s$ of $E^*$.
If $(E,h)$ satisfies the optimal $L^p$-extension condition for some $p>0$,
then $(E,h)$ is Griffiths semi-positive.
\end{thm}

A first related result in this direction was given by Guan-Zhou in \cite{GZh15d},
where they showed that Berndtsson's plurisubharmonic variation of the relative Bergman kernels \cite{Bob06}
can be deduced from the optimal $L^2$-extension condition.
By developing Guan-Zhou's method, Hacon-Popa-Schnell in \cite{HPS16} proved that a Hermitian vector bundle $(E,h)$ with singular Hermitian metric
is Griffiths semi-positive if it satisfies the so called minimal extension condition, a notion defined there as mentioned above.
Theorem \ref{thm: optimal Lp extension : Griffiths positive-intr} is proved by combining the ideas in \cite{GZh15d,HPS16}
and a lemma in \cite{DNW1}.

\begin{thm}\label{thm: multiple coarse Lp: Griffiths positivity-intr.}
Let $E$ be a holomorphic vector bundle over a domain $D\subset\mc^n$,
and $h$ be a singular Finsler metric on $E$,
 such that $|s|_{h^*}$ is upper semi-continuous
for any local holomorphic section $s$ of $E^*$.
If $(E,h)$ satisfies the multiple coarse  $L^p$-extension condition for some $p>0$, then $(E,h)$ is Griffiths semi-positive.
\end{thm}
Theorem \ref{thm: multiple coarse Lp: Griffiths positivity-intr.} was originally proved in \cite{DWZZ1}.
In this paper, we give a new proof based on the idea in the proof of \cite[Theorem 1.5]{DNW1}.

In the case for trivial line bundles over bounded domains,
Theorem \ref{thm:theta-nakano text_intr}-\ref{thm: optimal Lp extension : Griffiths positive-intr} were proved in \cite{DNW1}.
We should emphasize that Theorem \ref{thm:theta-nakano text_intr}-\ref{thm: multiple coarse Lp: Griffiths positivity-intr.}
are true for vector bundles of infinite rank, as well as for those of finite rank.

\begin{rem}
In applications, it is possible to prove that $(E, e^{\phi}h)$ satisfies the optimal $L^p$-extension condition for some $\phi\in C^0(D)$,
by Theorem \ref{thm:coarse estimate text-intr} /Theorem \ref{thm: optimal Lp extension : Griffiths positive-intr}/Theorem \ref{thm: multiple coarse Lp: Griffiths positivity-intr.}, which implies that
$$i\Theta_E\geq i\partial\bar\partial\phi\otimes Id_E$$
in the sense that $i\partial\bar\partial \log|s|^2_{h^*}\geq i\partial\bar\partial\phi$
in the sense of currents, for any nonvanishing local holomorphic section $s$ of $E^*$.
\end{rem}

We now explain why Theorem \ref{thm:theta-nakano text_intr}-\ref{thm: multiple coarse Lp: Griffiths positivity-intr.} can be
roughly viewed as converses of $L^2$-estimate  of $\bar\partial$ and Ohsawa-Takegoshi type $L^2$-extensions.

Let $(X,\omega)$ be a weakly pseudoconvex K\"ahler manifold, and $(E,h)$ be a Hermitian holomorphic vector bundle over $X$.
If the curvature of $(E,h)$ is Nakano semi-positive, then $(E,h)$ satisfies the optimal $L^2$-estimate condition
and the multiple coarse $L^2$-estimate condition by works of H\" ormander \cite{Hor65} and Demailly \cite{Dem82}.
Combining Theorem \ref{thm:theta-nakano text_intr}, we see in this setting that \emph{Nakano positivity for $(E,h)$ is equivalent to
the optimal $L^2$-estimate condition}.
Considering Theorem \ref{thm:coarse estimate text-intr} which shows that
the multiple coarse $L^2$-estimate condition for $(E,h)$ implies Griffiths positivity of $(E,h)$,
it is natural to propose the following
\begin{problem}
Does the multiple coarse $L^2$-estimate condition of $(E,h)$ imply the Nakano positivity of $(E,h)$?
\end{problem}

Let $(E,h)$ be a Hermitian holomorphic vector bundle over a bounded pseudoconvex domain $D$.
If the curvature of $(E,h)$ is Nakano semi-positive,
then $(E,h)$ satisfies the multiple coarse $L^2$-extension condition by Ohsawa-Takegoshi \cite{OT1},
and satisfies the optimal $L^2$-extension condition by B\l ocki \cite{Bl} and Guan-Zhou \cite{GZh12}\cite{GZh15d}.
Theorem \ref{thm: optimal Lp extension : Griffiths positive-intr} and Theorem \ref{thm: multiple coarse Lp: Griffiths positivity-intr.}
show that the optimal $L^2$-extension condition and the multiple coarse $L^2$-extension condition imply
Griffiths positivity of $(E,h)$.
It is observed in \cite{HI} that the optimal $L^2$-extension condition and the multiple coarse $L^2$-extension condition
are strictly weaker than the Nakano positivity in some cases.
Motivated by this and the above theorems, we believe that the optimal $L^2$-estimate condition and the multiple coarse $L^2$-estimate condition are equivalent to the Nakano positivity,
while, the optimal $L^2$-extension condition and the multiple coarse $L^2$-extension condition are equivalent to the Griffiths positivity.
So we propose the following problem.

\begin{problem}
Does the Griffiths positivity imply the optimal $L^2$-extension condition and the multiple coarse $L^2$-extension condition?
\end{problem}

\begin{rem}
By the proof of \cite[Proposition 0.2]{Ber-Pau10},
one can see that the optimal $L^2$-extension condition and the multiple coarse $L^2$-extension condition imply
the optimal $L^p$-extension condition and the multiple coarse $L^p$-extension condition for $0<p<2$.
\end{rem}

The second part of this paper is to apply the above theorems to study the curvature positivity of
direct image sheaves of twisted relative canonical bundles associated to holomorphic fibrations,
which is a topic of extensive study in recent years (see \cite{Bob06, Bob09a, BP08, Ber-Pau10, LS14, GZh15d, PT18, Cao141, HPS16, ZZ17, ZhouZhu,  Bob18, DWZZ1}).
The main novelty here is that Theorem \ref{thm:theta-nakano text_intr} provides a natural explanation and a very simple
unified proof of the Nakano positivity of direct image bundles associated to families of both Stein manifolds and compact K\"ahler manifolds,
with an effective estimate of the lower bound of the curvatures.

We first consider a family of bounded domains.
Let $\Omega=U\times D\subset \mc^n_t\times\mc^m_z$ be a bounded pseudoconvex domain and $p:\Omega\ra U$ be the natural projection.
Let $h$ be a Hermitian metric on the trivial bundle $E=\Omega\times\mc^r$ that is $C^2$-smooth to $\overline\Omega$.
For $t\in U$, let
$$F_t:=\{f\in H^0(D, E|_{\{t\}\times D}):\|f\|^2_{t}:=\int_D|f|^2_{h_t}<\infty\}$$
and $F:=\coprod_{t\in U}F_t$.
Since $h$ is continuous to $\overline\Omega$, $F_t$ are equal for all $t\in U$ as vector spaces.
We may view $(F, \|\cdot\|)$ as a trivial holomorphic Hermitian vector bundle of infinite rank  over $U$.

\begin{thm}\label{thm: direc im stein-intr}
Let $\theta$ be a continuous real $(1,1)$-form on $U$ such that $i\Theta_{E}\geq p^*\theta\otimes Id_E$,
then $i\Theta_{F}\geq \theta\otimes Id_F$ in the sense of Nakano.
In particular, if $i\Theta_{E}>0$ in the sense of Nakano, then $i\Theta_{F}>0$ in the sense of Nakano.
\end{thm}

Let $\pi:X\rightarrow U$ be a proper holomorphic submersion from  a K\" ahler manifolds $X$ of complex dimension $m+n$,  to a bounded pseudoconvex domain $U$, and $(E,h)$ be a Hermitian holomorphic vector bundle over $X$, with  Nakano semi-positive  curvature.
From the Ohsawa-Takegoshi extension theorem, the direct image $F:=\pi_*(K_{X/U}\otimes E)$ is a vector bundle,
whose fiber over $t\in U$ is $F_t=H^0(X_t, K_{X_t}\otimes E|_{X_t})$.
There is a hermtian metric $\|\cdot\|$ on $F$ induced by $h$:
 for any $u\in F_t$,
 $$\|u(t)\|^2_t:=\int_{X_t}c_mu\wedge \bar u,$$
where $m=\dim X_t$, $c_m=i^{m^2}$, and $u\wedge \bar u$ is the composition of the wedge product and the inner product on $E$.
So we get a Hermitian holomorphic vector bundle $(F, \|\cdot\|)$ over $U$.

\begin{thm}\label{thm: direct image-optimal L2 estimate-intr}
The Hermitian holomorphic vector bundle $(F, \|\cdot\|)$ over $U$ defined above
satisfies the optimal $L^2$-estimate condition.
Moreover, if $i\Theta_{E}\geq p^*\theta\otimes Id_E$ for a continuous real $(1,1)$-form $\theta$ on $U$,
then $i\Theta_{F}\geq \theta\otimes Id_F$ in the sense of Nakano.
\end{thm}

Theorem \ref{thm: direc im stein-intr} in the case that $E$ is a line bundle is a result of Berndtsson \cite[Theorem 1.1]{Bob09a},
the case for vector bundles $E$ without lower bound estimate was proved by Raufi \cite[Theorem 1.5]{Raufi 13} with the same method of Berndtsson.
Theorem \ref{thm: direct image-optimal L2 estimate-intr} in the case that $L$ is a line bundle is due to Berndtsson \cite{Bob09a},
and the case for vector bundles was proved in \cite{MoTa08} and \cite{LiYa14} by developing the method of Berndtsson.

Our method to Theorem \ref{thm: direc im stein-intr} and Theorem \ref{thm: direct image-optimal L2 estimate-intr} is very different.
In fact, taking Theorem \ref{thm:theta-nakano text_intr} for granted,
one can clearly see why Theorem \ref{thm: direc im stein-intr} and Theorem \ref{thm: direct image-optimal L2 estimate-intr} should be true,
since it is obvious that the bundles $F$ in Theorem \ref{thm: direc im stein-intr} and Theorem \ref{thm: direc im stein-intr}
satisfy the optimal $L^2$-estimate condition by H\"ormander's $L^2$-estimate of $\bar\partial$.

In this paper, we also provide some new methods to show that $(F, \|\cdot\|)$ is Griffiths semi-positive,
via Theorem \ref{thm:coarse estimate text-intr}, \ref{thm: optimal Lp extension : Griffiths positive-intr},
and \ref{thm: multiple coarse Lp: Griffiths positivity-intr.}.
By applying the tensor-power technique introduced in \cite{DWZZ1},
we show that $(F, \|\cdot\|)$ satisfies the multiple coarse $L^2$-estimate condition;
by applying the Ohsawa-Takegoshi extension theorem with optimal estimate for vector bundles (\cite{GZh15d}, \cite{ZhouZhu192}),
we show that $(F, \|\cdot\|)$ satisfies the optimal $L^2$-extension condition;
and by applying the tensor-power technique mentioned above and the Ohsawa-Takegoshi extension theorem,
we show that $(F, \|\cdot\|)$ satisfies the multiple coarse $L^2$-extension condition.\\

$\mathbf{Acknowlegements.}$
The authors are partially supported respectively by NSFC grants (11871451,  11801572, 11701031, 11688101).
The first author was partially supported by the University of Chinese Academy of Sciences.
The third author  is partially supported by Beijing Natural Science Foundation (Z190003, 1202012).

\section{Preliminaries}
\subsection{An extension property of Hermitian metrics on a line bundle}

In this section, we present a basic property of  K\" ahler manifolds, which admit positive Hermitian holomorphic line bundles.

 \begin{prop}\label{prop: construct function 1}Let $X$ be a K\" ahler  manifold, which admits a positive Hermitian holomorphic line bundle, and $(A,h_A)$ be a positive Hermitian holomorphic line bundle over $X$. Let $(U\subset X, z=(z_1,\cdots, z_n))$ be a coordinate chart on $X$, such that $A|_U$ is trivial, and $B\Subset U$ be a coordinate ball. Then for any smooth strictly plurisubharmonic function $\psi$ on $U$, there is a  positive integer $m$, and a Hermitian metric $h_m$ on the line bundle $A^{\otimes m}$, such that $h_m=e^{-\psi_m}$ on $U$ with $\psi_m|_B=\psi$.
 \end{prop}

 \begin{proof}
 Assume that $h_A|_U=e^{-\phi}$ for some smooth strictly plurisubharmonic function $\phi$ on $U$.
 We may assume that $\phi>0$.
 We may assume that $B=B_1$ is the unit ball, and the ball $B_{1+3\delta}$ with radius $1+3\delta$ is also contained in $U$, for  $0<\delta\ll 1$.
  Let $\chi$ be a cut-off function on $U$,
 such that $\chi$ is identically equal to $1$ near $\overline B_{1+\delta}$ and vanishes outside  $B_{1+2\delta}$.
 Let $\phi_{m}:=m\phi+\chi\log(\|z\|^2-1)$ on $U\backslash B_{1}$,
 where $m\gg 1$ is an integer such that $\phi_m$ is strictly p.s.h on $U$
 and $\phi_{m} >\psi$ on $\partial B_{1+\delta}$.

 Now we define a function $\psi_m$ on $U$ as follows:

 \begin{align*}
 \psi_m=\left\{
 \begin{array}{ll}
 \phi_m, & \hbox{outside $B_{1+\delta}$;} \\
  \max_\epsilon\{\phi_{m}, \psi\} , & \hbox{on $B_{1+\delta}\setminus B_{1}$;}\\
  \psi, & \hbox{on $B_{1}$.}
 \end{array}
 \right.
 \end{align*}

Then for $0<\epsilon\ll 1$, $\psi_m$ is strictly p.s.h on $U$, $\psi_m|_B=\psi$,
and equals to $m\phi$ on $U\backslash B_{1+2\delta}$.
So $\psi_m$ gives a Hermitian metric on $A^{\otimes m}|_U$ which coincides with $h^{\otimes m}$ on $U\backslash B_{1+2\delta}$.
 \end{proof}
 \subsection{Basics of Hermitian holomorphic vector bundles}

 Let $(X,\omega)$ be a complex manifold of complex dimension $n$, equipped with a Hermitian metric $\omega$,
 and $(E,h)$ be a Hermitiann holomorphic vector bundle of rank $r$ over $X$.
 In this subsection, we assume  $r<\infty$.

 Let $D=D'+\bar{\partial}$ be the Chern connection of $(E,h)$, and $\Theta_{E,h}=[D',\bar\partial]=D'\bar{\partial}+\bar{\partial}D'$ be the Chern curvature tensor.
 Denote by $(e_1,\cdots, e_r)$ an orthonormal frame of $E$ over a coordinate patch $\Omega\subset X$ with complex coordinates $(z_1,\cdots, z_n)$, and
 $$i\Theta_{E,h}=i\sum_{1\leq j,k\leq n,1\leq \lambda,\mu\leq r}c_{jk\lambda\mu}dz_j\wedge d\bar z_k\otimes e^*_\lambda\otimes e_{\mu}, ~~\bar c_{jk\lambda\mu}=c_{kj\mu\lambda}.$$
 To $i\Theta_{E,h}$ corresponds a natural Hermitian form $\theta_{E,h}$ on $TX\otimes E$ defined by
 \begin{align}\label{eqn: chern curvature form}
 \theta_{E,h}(u,u)=\sum_{j,k,\lambda,\mu}c_{jk\lambda\mu}(x)u_{j\lambda}\bar u_{k\mu}, ~~~~~u\in T_xX\otimes E_x
 .\end{align}

 \begin{defn}\label{defn:positivities}
 \begin{itemize}
 \item  $E$ is said to be Nakano positive (resp. Nakano semi-positive) if $\theta_{E,h}$ is positive (resp. semi-positive) definite  as a Hermitian form on $TX\otimes E$, i.e. for every $u\in TX\otimes E$, $u\neq 0$, we have $$\theta(u,u)>0 ~~~(\mbox{resp.} \geq 0).$$
 \item $E$ is said to be Griffiths positive (resp. Griffiths semi-positive) if for any $x\in X$, all $\xi\in T_xX$ with $\xi\neq 0$, and $s\in E_x$ with $s\neq 0$, we have
 $$\theta(\xi\otimes s,\xi\otimes s)>0~~~(\mbox{resp.} \geq 0).$$
 \item Nakano negative (resp. Nakano semi-negative) and Griffiths negative (resp. Griffiths semi-negative) are similarly defined by replacing $>0$ (resp. $\geq 0$) by $<0$ (resp. $\leq 0$) in the above definitions respectively.
 \end{itemize}
 \end{defn}

 \begin{rem} The following are  basic facts about Griffiths positivity and Nakano positivity.
 \begin{itemize}
 \item It is a well-known fact that, a Hermitian holomorphic vector  bundle $(E,h)$ is Griffiths positive (resp. semi-positive) if and only if $(E^*,h^*)$ is Griffiths negative (resp. semi-negative). However, Nakano positivity does not share this duality condition, see \cite[Chapter VII, Page 339, Example 6.8]{Dem} for an example.
\item It is a fact that, Griffiths positivity can be explained as a several complex variables property, see Definition \ref{def:positivity of finsler} in \S \ref{subsec:finsler metric}. However, Nakano positivity does not have such a characterization.
\end{itemize}
\end{rem}

  \begin{rem}\label{rem: Nakano tensor product} Let $(E_1,h_1) $ and $(E_2,h_2)$ be two Hermitian holomoprhic vector bundles over a complex $n$-dimensional manifold $X$. It is a basic fact that the Chern connection $D_{E_1\otimes E_2}$ of $(E_1\otimes E_2, h_1\otimes h_2)$ is just $D_{E_1}\otimes \text{Id}_{E_2}+\text{Id}_{E_1}\otimes D_{E_2}$, and we have the following Chern curvature formula

  $$\Theta_{E_1\otimes E_2, h_1\otimes h_2}=\Theta_{E_1,h_1}\otimes \text{Id}_{E_2}+\text{Id}_{E_1}\otimes \Theta_{E_2,h_2}.$$
\end{rem}
From \eqref{eqn: chern curvature form} and Remark \ref{rem: Nakano tensor product}, we get the following

 \begin{lem}\label{lem: Nakano tensor product}
 Let $(E_1,h_1)$ and $(E_2,h_2)$ be two Hermitian holomorphic vector bundles over a complex manifold $X$. Let $(E,h):=(E_1\otimes E_2, h_1\otimes h_2)$. Then if $(E_1, h_1)$ and $(E_2,h_2)$ are Nakano positive (resp. Nakano semi-positive), then  $(E,h)$ is Nakano positive (resp. Nakano semi-positive).
 \end{lem}

 \begin{lem}\label{lem: Nakano fiber tensor product}
 Let $\pi_i:X_i\rightarrow Y$ be two holomorphic submersions for $j=1,2$. Let $X\subset X_1\times X_2$ be the fiberwise product of $X_1$ and $X_2$ with respect to $\pi_1$ and $\pi_2$, and $pr_j:X\rightarrow X_j$ be the natural projections from $X$ to $X_j$  for  $j=1,2$. Let $(E_1,h_1)$ and $(E_2,h_2)$ be two  Hermitian  holomorphic vector bundles over $X_1$ and $X_2$, respectively. Denote by $(E, h)$ the Hermitian holomorphic vector bundle $(pr_1^*E_1\otimes pr_2^*E_2, pr_1^*h_1\otimes pr_2^* h_2)$ on $X$. If $(E_1,h_1)$ and $(E_2,h_2)$ are Nakano positive (resp. Nakano semi-positive), then $E$ is also Nakano positive (resp. Nakano semi-positive).
 \end{lem}


 For any $u\in \Lambda^{p,q}T^*_X\otimes E$, we consider the global $L^2$-norm
 \begin{align*}
 \|u\|^2=\int_X|u|_{\omega,h}^2dV_\omega,
 \end{align*}
 where $|u|_{\omega,h}$ is the pointwise Hermitian norm and $dV_\omega=\omega^n/n!$ is the volume form on $X$.  This $L^2$-norm induces an $L^2$-inner product on $\Lambda^{p,q}T^*_X\otimes E$, and thus we can define $D'^*$ and $\bar{\partial}^*$ operators as the (formal) adjoint of $D'$ and $\bar{\partial}$, respectively. Let $$\Delta'=D'D'^*+D'^*D', \ \ \Delta''=\bar{\partial}\bar{\partial}^*+\bar{\partial}^*\bar{\partial}$$ be the corresponding $D'$ and $\bar{\partial}$-Laplace operators.

 \begin{lem}[Bochner-Kodaira-Nakano identity, see \cite{Dem}]\label{lem: BKN identity}
 Let $(X,\omega)$ be a K\"{a}hler manifold, $(E,h)$ be a Hermitian vector bundle over $X$. The complex Laplacian operators
 $\Delta'=D'D'^*+D'^*D'$ and $\Delta''=\bar{\partial}\bar{\partial}^*+\bar{\partial}^*\bar{\partial}$ acting on $E$-valued forms satisfy the identity
 $$\Delta''=\Delta'+[i\Theta_{E,h},\Lambda_\omega].$$
 \end{lem}

 Let us say more on the Hermitian operator $[i\Theta_{E,h},\Lambda_\omega]$.  Let $x_0\in X$ and $(z_1,\cdots, z_n)$ be local coordinates centered at $x_0$, such that $(\partial/\partial z_1,\cdots, \partial/\partial z_n)$ is an orthonormal basis of $TX$ at $x_0$. One can write $$\omega=i\sum dz_j\wedge d\bar z_j+O(\|z\|),$$  and $$i\Theta_{E,h}(x_0)=i\sum_{j,k,\lambda,\mu}c_{jk\lambda\mu}dz_j\wedge d\bar z_k \otimes e^*_\lambda\otimes e_\mu,$$
 where $(e_1,\cdots, e_r)$ is an orthonormal basis of $E_{x_0}$. Let $u=\sum{u_{K,\lambda}}dz\wedge d\bar z_K\otimes e_\lambda\in \Lambda^{n,q}T^*_X\otimes E$ , where $dz=dz_1\wedge \cdots\wedge dz_n$.  In \cite[Chapter VII, Page 341, (7.1)]{Dem}, it is computed that
 \begin{align}\label{eqn: computation of B 1}
 \langle [i\Theta_{E,h},\Lambda_\omega]u,u\rangle=\sum_{|S|=q-1}\sum_{j,k,\lambda,\mu}c_{jk\lambda\mu}u_{jS,\lambda}\bar u_{kS,\mu}.
 \end{align}
 In particular, if $q=1$, \eqref{eqn: computation of B 1} becomes
 \begin{align}\label{eqn: computation of B 2}
 \langle [i\Theta_{E,h},\Lambda_\omega]u,u\rangle=\sum_{j,k,\lambda,\mu}c_{jk\lambda\mu}u_{j,\lambda}\bar u_{k,\mu}
 \end{align}
 Comparing \eqref{eqn: chern curvature form} and \eqref{eqn: computation of B 2}, we obtain the following
 \begin{lem}\label{lem: Hermitian operator formula}
 Let $(X,\omega)$ be a K\"{a}hler manifold, $(E,h)$ be a Hermitian vector bundle over $X$. Then $(E,h)$ is Nakano positive (resp. semo-positive) if and only if the Hermitian operator $[i\Theta_{E,h},\Lambda_\omega]$ is positive definite (resp. semi-positive definite) on $\Lambda^{n,1}T^*_X\otimes E$.
 \end{lem}

\subsection{Basic concepts and conditions of Hermitian  vector bundles of infinite rank}\label{subsec:infinite rank v.b.}
In this subsection, we will briefly discuss some concepts and basic conditions
of Hermitian holomorphic vector bundles of infinite rank,
and explain why the above mentioned Bochner-Kodaira-Nakano identity also holds in this framework.

Let $H$ be a Hilbert space (separable, say over $\mc$) with inner product $(,)$. Let $U\subset \mathbb R^n$ be open.
Let $f:U\rightarrow H$ be a map. If
\begin{align*}
\frac{\partial f}{\partial x_j}=\lim_{\Delta x_j\rightarrow 0}\frac{f(x_1,\cdots,x_j+\Delta x_j,\cdots, x_{n})-f(x_1,\cdots,x_{n})}{\Delta x_j}\in H
\end{align*}
exists and is continuous on $U$ for any $j=1,2,\cdots,n$,
$f$ is called of $C^1$.
We say $f$ is of $C^r$  if all partial derivatives $\frac{\partial^rf}{\partial x_1^{r_1}\cdots \partial x_{n}^{r_{n}}}: U\rightarrow H$ of order $r$ exist and continuous,
and $f$ is smooth if $f$ is of $C^r$ for any $r$.

Let $\{e_\lambda\}^{\infty}_{\lambda=1}$ be an orthonormal basis of $H$.
Then a map $f:U\ra H$ can be written as $(f_1, f_2,\cdots)$, where $f_\lambda$ are functions on $U$ such that
$$\|f(z)\|^2=\sum_\lambda|f_\lambda(z)|^2.$$
If $f$ is continuous, by Dini's theorem, one can see that the series $\sum_\lambda|f_\lambda(z)|^2$
converges uniformly locally on $U$ to $\|f(z)\|^2$;
and if $f$ is smooth, then
$\sum_\lambda |\frac{\partial^r f_\lambda}{\partial x_1^{r_1}\cdots \partial x_n^{r_n}}|^2$ locally uniformly converges to $\|\frac{\partial^r f}{\partial x_1^{r_1}\cdots \partial x_n^{r_n}}\|^2$.

Now assume $U\subset\mc^n$ be an open set.
A map $f:U\ra H$ is called holomorphic if $f$ is smooth and satisfies the Cauchy-Riemann equation
$$\frac{\partial}{\partial\bar z_j}f:=\frac{1}{2}(\frac{\partial}{\partial x_j}+i\frac{\partial}{\partial y_j})f=0,\ j=1, \cdots, n.$$

We now consider holomorphic vector bundles of infinite rank with $H$ as the model of the fibers.
In the present paper, we will focus on local conditions of holomorphic vector bundles,
So we just consider the trivial bundle $E:=U\times H\rightarrow U$,
here we view $H$ as a locally convex topological vector space.

\begin{defn}
A Hermitian metric on $E$ is a map
\begin{align*}
h:U\rightarrow Herm(H)
\end{align*}
which satisfies the following conditions:
\begin{itemize}
\item [(1)] $h$ is smooth, and
\item [(2)] $h(z)\geq \delta(z) Id$ for some positive continuous function $\delta$ on $U$,
\end{itemize}
where $Herm(H)$ is the space of self-adjoint bounded operators on $H$.
\end{defn}

Given $h$ as above, we get a smooth family of inner products on $H$ as
$$(u,v)_z=(h(z)u,v),\ z\in U.$$
So our definition of the Hermitian metric on $E$ matches to the definition of Hermitian metrics for holomorphic vector bundles
of finite rank.

Given a Hermitian metric $h$ on $E$, we can define a unique connection $D=D'+\bar\partial$ on $E$ which is compatible with $h$ and whose $(0,1)$-part is $\bar\partial$,
as in the finite rank case. We view a section of $E$ as a map from $U$ to $H$.
Assume $u$ is a smooth section of $E$, and $v\in H$ viewed as a constant section of $E$,
then from the condition that
$$\partial (hu,v)=(hD'u,v),$$
we get
$$D'u=h^{-1}\partial (hu).$$
This formula shows that $D'u$ is a smooth section of $\Lambda^{1,0}T^*U\otimes E$.

The curvature operator of $(E,h)$ is give by
$$\Theta_{E,h}=[D',\bar\partial],$$
which is an operator that maps smooth sections of $E$ to smooth sections of $\Lambda^{1,1}T^*U\otimes E$.
In the same way as in the case of finite rank vector bundles,
Nakano positivity and Griffiths positivity can be defined for $(E,h)$.

We now show that, at any point $z_0\in U$, the metric $h$ coincides with a flat metric up to order 1.
To show this,  we may assume $z_0=0$ and $h_0=Id$.
Let $h_j=\frac{\partial h}{\partial z_j}(0),\ j=1,\cdots, n$, and define
$$\tilde e_\lambda=e_\lambda-\sum_j z_jh_j(e_\lambda),$$
then
$$(\tilde{e}_{\lambda},\tilde{e}_{\mu})_z=\delta_{\lambda\mu}+O(\|z\|^2),$$
where the bound for $O(\|z\|^2)$ is uniform for $\lambda, \mu$.

With the above preparation,
following the line of the proof of \cite[Theorem 1.1, Theorem 1.2, Chapter VII, \S 1]{Dem},
we see that the Bochner-Kodaira-Nakano identity also holds for $(E,h)$.
\subsection{Singular Finsler metrics on holomorphic vector bundles}\label{subsec:finsler metric}

In this subsection, we recall the notions of singular Finsler metrics on holomorphic vector bundles and  positively curved singular Finsler metrics on coherent analytic sheaves,  introduced in \cite{DWZZ1}, see also \cite{DWZZ2}.

 \begin{defn}\label{def:finsler on v.b.} Let $E\rightarrow X$ be a holomorphic vector bundle over a complex manifolds $X$. A (singular) Finsler metric $h$ on $E$ is a function $h:E\rightarrow [0, +\infty]$, such that $|cv|^2_h:=h(cv)=|c|^2h(v)$ for any $v\in E$ and $c\in \mathbb C$.
 \end{defn}

 \begin{defn}
 For a singular Finsler metric $h$ on $E$, its dual Finsler metric $h^*$ on the dual bundle $E^*$ of $E$ is defined as follows. For $f\in E^*_x$, the fiber of $E^*$ at $x\in X$, $|f|_{h^*}$ is defined to be $0$ if $|v|_h=+\infty$ for all nonzero $v\in E_x$; otherwise, $$|f|_{h^*}:=\sup\{|f(v)|: v\in E_x,  |v|_h\leq 1\}\leq +\infty.$$
 \end{defn}

 \begin{defn}\label{def:positivity of finsler}
 Let $(E,h)$ be a holomorphic vector bundle over a complex manifold $X$, equipped with a singular Finsler metric $h$. We call $h$ is negatively curved (in the sense of Griffiths) if for any local holomorphic section $s$ of $E$, the function $\log|s|^2_h$ is plurisubharmonic, and we call $h$ is positively curved (in the sense of Griffiths) if its dual metric $h^*$ is negatively curved.
 \end{defn}

 \begin{defn}\label{def:finsler on sheaf}
 Let $\mathcal F$ be a coherent analytic sheaf on a complex manifold $X$. Let $Z\subset X$ be an analytic subset of $X$ such that $\mathcal F|_{X\setminus Z}$ is locally free.  A positively curved singular Finsler metric $h$ on $\mathcal F$ is a singular Finsler metric on the holomorphic vector bundle $\mathcal F|_{X\setminus Z}$, such that for any local holomorphic section $s$ of the dual sheaf $\mathcal F^*$ on an open set $U\subset X$, the functionn $\log |s|_{h^*}$ is plurisubharmonic on $U\setminus Z$, and can be extended to a plurisubharmonic function on $U$.
 \end{defn}

 \begin{rem}
 Suppose that $\log|s|_{h^*}$  is p.s.h. on $U\setminus Z$.
 It is well-known that if  codim$_{\mathbb{C}}(Z)\geq 2$ or $\log|s|_{H^*}$ is locally bounded above near $Z$,
 then $\log|s|_{h^*}$ extends across $Z$ to $U$ uniquely as a p.s.h function.
 Definition \ref{def:finsler on sheaf} matches Definition \ref{def:finsler on v.b.}
 and Definition \ref{def:positivity of finsler} if $\mathcal F$ is a vector bundle.
 \end{rem}

\subsection{$L^2$ theory of $\bar\partial$}
In this section, we  recall H\" ormander's $L^2$-estimate of $\bar\partial$ and Ohsawa-Takegoshi type $L^2$-extension of holomorphic sections, for holomorphic vector bundles.

We first clarify some notions and notations.
Let $H$ be a Hilbert space with an inner product $(\cdot, \cdot)$, and $A:H\ra H$ be a bounded semi-positive self-adjoint operator with closed range $Im A$.
The we have an orthogonal decomposition
$$H=Im A\oplus \ker A$$
and $A|_{Im A}:Im A\ra Im A$ is a linear isomorphism.
In the remaining of the paper, we always denote $A|^{-1}_{Im A}$ by $A^{-1}$,
as in general references about complex geometry,
and define $(A^{-1}v, v)=+\infty$ if $v\not\in ImA$.

\begin{lem}[c.f. {\cite[Theorem 4.5]{Dem}}]\label{thm: L2 estimate Nakano}Let $(X,\omega)$ be a complete K\" ahler manifold, with a K\" ahler metric which is not necessarily complete.  Let $(E,h)$ be a Hermitian  vector bundle of rank $r$ over $X$, and assume that the curvature operator $B:=[i\Theta_{E,h},\Lambda_\omega]$ is semi-positive definite everywhere on $\Lambda^{p,q}T_X^*\otimes E$, for some $q\geq 1$. Then for any form $g\in L^2(X,\Lambda^{p,q}T^*_{X}\otimes E)$ satisfying $\bar{\partial}g=0$ and $\int_X\langle B^{-1}g,g\rangle dV_\omega<+\infty$, there exists $f\in L^2(X,\Lambda^{p,q-1}T^*_X\otimes E)$ such that $\bar{\partial}f=g$ and $$\int_X|f|^2dV_\omega\leq \int_X\langle B^{-1}g,g\rangle dV_\omega.$$
\end{lem}

The following $L^2$-extension theorem for K\"ahler families is due to Zhou-Zhu {\cite[Theorem 1.1]{ZhouZhu192}}. The same result for  projective families is due to Guan-Zhou \cite{GZh12, GZh15d}.
\begin{lem}[{\cite[Theorem 1.1]{ZhouZhu192}}]\label{thm: optimal L2 extension}
Let $\pi:X\rightarrow B$ be a proper holomorphic submersion from a  complex $n$-dimensional K\" ahler manifold $(X,\omega)$ onto a unit ball in $\mathbb C^m$. Let  $(E,h=h_E)$ be a Hermitian holomorphic   vector bundle over $X$, such that the curvature  $i\Theta_{E,h_E}\geq 0$ in  the sense of Nakano. Let $t_0\in B$ be an arbitrarily fixed point. Then for every section $u\in H^0(X_{t_0}, K_{X_{t_0}}\otimes E|_{X_{t_0}})$, such that
\begin{align*}
\int_{X_{t_0}}|u|_{\omega, h}^2dV_{\omega_{X_{t_0}}}<+\infty,
\end{align*}
there is a section $\widetilde u\in H^0(X,K_X\otimes E)$, such that $\widetilde u|_{X_{t_0}}=\widetilde u\wedge dt$, with the following $L^2$-estimate
\begin{align*}
\int_X|\widetilde u|^2_{\omega,h}dV_{X,\omega}\leq \mu({B})\int_{X_{t_0}}|u|_{\omega, h}^2dV_{\omega_{X_{t_0}}},
\end{align*}
where $dt=dt_1\wedge\cdots\wedge dt_m$, and $t=(t_1,\cdots, t_m)$ are  the holomorphic coordintes on $\mathbb C^m$, and $\mu(B)$ is the volume of the unit ball in $\mathbb C^m$ with respect to the Lebesgue measure on $\mathbb C^m$.
\end{lem}
\begin{rem}

We take $R(t)=e^{-t}$, $\alpha_0=\alpha_1=0$, and $\psi=m\log\|t-t_0\|^2$ in \cite[Theorem 1.1]{ZhouZhu192}, and from \cite[Lemma 4.14]{GZh15d}, \cite[Remark 1.2]{ZhouZhu192}, we can get the precise form of Theorem \ref{thm: optimal L2 extension}.
\end{rem}

\section{Positivities of holomorphic vector bundles via $L^p$-conditions of $\bar\partial$}
The aim of this section is to prove Theorem \ref{thm:theta-nakano text_intr}-- Theorem \ref{thm: multiple coarse Lp: Griffiths positivity-intr.}.

\subsection{Characterizations of Nakano positivity in term of optimal $L^2$-estimate condition}
\begin{thm}[=Theorem \ref{thm:theta-nakano text_intr}]\label{thm:theta-nakano text}
Let $(X,\omega)$ be a  K\"{a}hler manifold of dimension $n$ with a K\" ahler metric $\omega$, which admits a positive Hermitian holomorphic line bundle,
$(E,h)$ be a smooth Hermitian vector bundle over $X$,
and $\theta\in C^0(X,\Lambda^{1,1}T^*_X\otimes End(E))$ such that $\theta^*=\theta$.
If for any $f\in\mathcal{C}^\infty_c(X,\wedge^{n,1}T^*_X\otimes E\otimes A)$ with $\bar\partial f=0$,
and any positive Hermitian line bundle $(A,h_A)$ on $X$ with $i\Theta_{A,h_A}\otimes Id_E+\theta>0$ on $\text{supp}f$,
there is $u\in L^2(X,\wedge^{n,0}T_X^*\otimes E\otimes A)$, satisfying $\bar\partial u=f$ and
$$\int_X|u|^2_{h\otimes h_A}dV_\omega\leq \int_X\langle B_{h_A,\theta}^{-1}f,f\rangle_{h\otimes h_A} dV_\omega,$$
provided that the right hand side is finite,
where $B_{h_A,\theta}=[i\Theta_{A,h_A}\otimes Id_E+\theta,\Lambda_\omega]$,
then $i\Theta_{E,h}\geq\theta$ in the sense of Nakano.
On the other hand, if in addition $X$ is assumed to have a complete K\"ahler metric,
the above condition is also necessary for that $i\Theta_{E,h}\geq\theta$ in the sense of Nakano.
In particular, if $(E,h)$ satisfies the optimal $L^2$-estimate condition, then $(E,h)$ is Nakano semi-positive.
\end{thm}
\begin{proof}
The second statement is a corollary of Theorem \ref{thm: L2 estimate Nakano}.
We now give the proof of the first statement.
We give the proof in the case that $\theta$ is $\mathcal C^1$,
and the general case follows the proof by an approximation argument.

To illustrate the main idea more clearly, we may assume that  there is a  smooth strictly plurisubharmonic function on $X$, which corresponds to the existence of a positive Hermitian trivial holomophic line bundle on $X$. For general case, the same proof goes through by replacing data related to  $e^{-\psi}$ by $h_A$, and using Proposition \ref{prop: construct function 1}.

Let $\psi$ be any smooth strictly plurisubharmonic function on $X$.
By assumption, we can solve the equation $\bar{\partial}u=f$  for any $\bar\partial$-closed $f\in\mathcal{C}^\infty_c(X,\wedge^{n,1}T^*_X\otimes E)$, with the estimate
$$\int_X|u|^2e^{-\psi}dV_\omega\leq \int_X\langle B_{\psi,\theta}^{-1} f,f\rangle e^{-\psi}dV_\omega,$$
where $B_{\psi,\theta}:=[i\partial\bar\partial \psi\otimes Id_E+\theta, \Lambda_\omega]$. For any $\alpha\in\mathcal{C}^\infty_c(X,\wedge^{n,1}T^*_X\otimes E)$, we have
\begin{align*}
|\langle\langle\alpha,f\rangle\rangle_{\psi}|&=|\langle\langle\alpha,\bar{\partial}u\rangle\rangle_{\psi}|\\
&=|\langle\langle {\bar{\partial}}^{*} \alpha,u\rangle\rangle_{\psi}|\\
&\leq||u||_{\psi}||{\bar{\partial}}^{*} \alpha||_{\psi},
\end{align*}
where $\bar\partial^*$ is the adjoint of $\bar\partial$ with respect to $\omega$, $e^{-\psi}h$.

From  Lemma \ref{lem: BKN identity}, we obtain
\begin{equation}\label{eq1}
\begin{split}
&|\langle\langle\alpha,f\rangle\rangle_{\psi}|^2\\
\leq& \int_ X\langle B_{\psi,\theta}^{-1}f,f\rangle e^{-\psi}dV_\omega\\
&\times\left(||D'\alpha||_\psi^2+||{D'}^*\alpha||_\psi^2+\langle\langle[i\Theta_{E,h}+
i\partial\bar\partial\psi\otimes Id_E,\Lambda_\omega]\alpha,\alpha\rangle\rangle_\psi-||{\bar{\partial}}\alpha||_\psi^2\right)\\
\leq& \int_ X\langle B_{\psi,\theta}^{-1}f,f\rangle e^{-\psi}dV_\omega\times\left(\langle\langle[i\Theta_{E,h}+
i\partial\bar\partial\psi\otimes Id_E,\Lambda_\omega]\alpha,\alpha\rangle\rangle_\psi+||{D'}^*\alpha||_\psi^2\right),
\end{split}
\end{equation}
where $D'$ is the $(1,0)$ part of the Chern connection on $E$ with respect to the metric $e^{-\psi}h$.

Let $\alpha=B_{\psi,\theta}^{-1}f$, i.e., $f=B_{\psi,\theta} \alpha.$
Then inequality (\ref{eq1}) becomes
\begin{equation*}
\begin{split}
&\left(\langle\langle B_{\psi,\theta} \alpha,\alpha\rangle\rangle_{\psi}\right)^2\\
\leq&\langle\langle B_{\psi,\theta} \alpha,\alpha\rangle\rangle_{\psi}
\left(\langle\langle[i\Theta_{E,h},\Lambda_\omega]\alpha,\alpha\rangle\rangle_\psi+
\langle\langle B_{\psi,0} \alpha,\alpha\rangle\rangle_{\psi}+||{D'}^*\alpha||_\psi^2\right).
\end{split}
\end{equation*}
Therefore, we can get
\begin{equation}\label{eq2}
\langle\langle[i\Theta_{E,h}-\theta,\Lambda_\omega]\alpha,\alpha\rangle\rangle_\psi+||{D'}^*\alpha||_\psi^2\geq0.
\end{equation}
We argue by contradiction.
Suppose that $i\Theta_{E,h}-\theta$ is not Nakano semi-positive on $X$.
By Lemma \ref{lem: Hermitian operator formula}, there is $x_0\in X$ and $\xi_0\in \Lambda^{n,1}T^*_{X,x_0}\otimes E_{x_0}$ such that $|\xi_0|=1$
and $\langle[i\Theta_{E,h}-\theta,\Lambda_{\omega}]\xi_0,\xi_0\rangle=-2c$ for some $c>0$.

Let $(U;z_1,z_2,\cdots,z_n)$ be a holomorphic coordinate on $X$ centered at $x_0$
such that $\omega=i\sum dz_j\wedge d\bar z_j+O(|z|^2)$,
and assume $\{e_1,e_2,\cdots,e_r\}$ is a holomorphic frame of $E$ on $U$.
Let $\xi=\sum \xi_{j\lambda}dz_1\wedge\cdots\wedge dz_n\wedge d\bar{z}_j\otimes e_\lambda$,
with constant coefficients such that $\xi(x_0)=\xi_0$.
We may assume
$$\langle[i\Theta_{E,h}-\theta,\Lambda_\omega]\xi,\xi\rangle<-c$$
on $U$.
Choose $R>0$ such that $B(0,R):=\{z:|z|<R\}\subset U$, and write $B(0,R)$ as $B_R$.

Choose $\chi\in \mathcal{C}^\infty_c(B_R)$, satisfying $\chi(z)=1$ for $z\in B_{R/2}$.
Let $f=\bar{\partial}\nu$ with
$$\nu(z)=(-1)^n \sum_{j,\lambda}\xi_{j\lambda}\bar{z}_j\chi(z)dz_1\wedge\cdots\wedge dz_n\otimes e_\lambda.$$
Then $$f(z)=\sum \xi_{j\lambda}dz_1\wedge\cdots\wedge dz_n\wedge d\bar{z}_j\otimes e_\lambda$$
for $z\in B_{R/2}$.
From Proposition \ref{prop: construct function 1}, we can construct a  smooth strictly plurisubharmonic function $\psi$ on $X$, such that
$\psi|_{B_R}(z)=|z|^2-\frac{R^2}{4}$.
For any integer $m>0$, set $$\psi_m(z)=m\psi(z).$$
As before, set
$\alpha_m=B_{\psi_m,\theta}^{-1}f=\frac{1}{m}B_{\psi,\theta/m}^{-1}f$.
By \cite[Chapter VII, Theorem 1.1]{Dem}, we have
$${D'}^*B_{\psi,0}^{-1}f(0)=0.$$
So after shrinking $R$, we can get $|{D'}^*\alpha_m(z)|\leq\frac{\sqrt c}{2m}$ for $z\in B_{R/2}$ and any $m$.
Since $f$ has compact support in $B_R$,
there is a constant $C>0$, such that $|\langle[i\Theta_{E,h}-\theta,\Lambda_\omega]\alpha_m,\alpha_m\rangle|\leq \frac{C^2}{m^2}$ and $|{D'}^*\alpha_m|\leq \frac{C}{m}$
hold for any $m>0$.

We now estimate both terms in \eqref{eq2} with $\alpha$ and $\psi$
replaced by $\alpha_m$ and $\psi_m$ defined as above.

\begin{equation}\label{eq4}
\begin{split}
&m^2\left(\langle\langle[i\Theta_{E,h}-\theta,\Lambda_\omega]\alpha_m,\alpha_m\rangle\rangle_{\psi_m}+||{D'}^*\alpha_m||_{\psi_m}^2\right)\\
=&m^2\left(\int_{B_{R/2}}\langle[i\Theta_{E,h}-\theta,\Lambda_\omega]\alpha_m,\alpha_m\rangle e^{-\psi_m}dV_\omega+
\int_{B_{R/2}}| {D'}^*\alpha_m|^2e^{-\psi_m}dV_\omega\right)\\
&+m^2\left(\int_{ B_R\setminus B_{R/2}} \langle[i\Theta_{E,h}-\theta,\Lambda_\omega]\alpha_m,\alpha_m\rangle e^{-\psi_m}dV_\omega
+\int_{ B_R\setminus B_{R/2}} | {D'}^*\alpha_m|^2e^{-\psi_m}dV_\omega\right)\\
\leq& -\frac{3c}{4}\int_{B_{R/2}}e^{-\psi_m}dV_\omega+2C^2\int_{ B_R\setminus B_{R/2}} e^{-\psi_m}dV_\omega.
\end{split}
\end{equation}
Since $\lim_{m\rightarrow+\infty}\psi_m(z)=+\infty$
for $z\in  B_R\setminus \overline{B}_{R/2}$, and $\psi_m(z)\leq 0$ for $z\in B_{R/2}$ and all $m$.
Therefore, we obtain from \eqref{eq4} that
$$\langle\langle[i\Theta_{E,h}-\theta,\Lambda_\omega]\alpha_m,\alpha_m\rangle\rangle_{\psi_m}+||{D'}^*\alpha_m||_{\psi_m}^2<0$$ for $m>>1$,
which contradicts to the inequality (\ref{eq2}).
\end{proof}

\begin{rem}
With the discussion in \S \ref{subsec:infinite rank v.b.}, the above proof holds for vector bundles of infinite rank.
\end{rem}

\subsection{Griffiths positivity in terms of multiple coarse $L^p$-estimate condition}
\begin{thm}[=Theorem \ref{thm:coarse estimate text-intr}]\label{thm:coarse estimate text}
Let $(X,\omega)$ be a K\"{a}hler manifold, which admits a positive Hermitian holomorphic line bundle,  and $(E,h)$ be a holomorphic vector bundle over $X$ with a continuous Hermitian metric $h$.
If $(E,h)$ satisfies the multiple coarse $L^p$-estimate condition for some $p>1$,  then $(E,h)$ is Griffiths semi-positive.
\end{thm}
\begin{proof}
We prove the theorem by modifying the idea in \cite{HI,DNW1}.
For the same reason as in the proof of Theorem \ref{thm:theta-nakano text},
we may assume that there is a strictly smooth plurisubharmonic function on $X$.

We will show that $(E,h)$ satisfies the multiple coarse $L^p$-extension condition.
We assume that $D', z=(z_1, \cdots, z_n)$ is an arbitrary coordinate chart on $X$,
and let $D$ be an arbitrary relatively compact subset of $D'$.
We assume that $E|_{D'}=D'\times\mc^r$ is trivial and $\omega|_D\leq C\  i/2\sum_{j=1}^{n}dz_j\wedge d\bar{z}_j$
with some $C>0$.

Fix  an integer $m>0$, $w\in D$ (we identify $w$ with its coordinate $z(w)$) and $a\in E_w$ with $|a|_h=1$.
We will construct $f \in  H^0(X,E^{\otimes m})$ such that $f(w) = a^{\otimes m}$ and
$$ \int_{X} |f|^p _{h^{\otimes m}}dV_\omega\leq C'_m ,$$
where $C'_m$ are uniform constants independent of $w$
that satisfy
$$\lim_{m\rightarrow\infty}\frac{\log C'_m}{m}=0.$$

Let $\chi = \chi(t)$ be a smooth function on $\mathbb{R}$, such that
\begin{itemize}
	\item $\chi(t) = 1$ for $t \leq 1/4$,
	\item $\chi(t) = 0$ for $t \geq 1$, and
	\item $|\chi'(t)|\leq 2$ on $\mathbb{R}$.
\end{itemize}
Viewing $a$ as a constant section of $E|_D$,
we define an $E^{\otimes m}$-valued  $(n,1)$-form $\alpha_\epsilon$ by
\begin{align*}
\alpha_\epsilon &:= \dbar\chi\left({|z-w|^2 \over \epsilon^2 }\right)dz \otimes a^{\otimes m}\\
&= \chi'\left(\frac{|z-w|^2}{\epsilon^2} \right) \sum_j \frac{z_j-w_j}{\epsilon^2} d\bar{z}_j\wedge dz\otimes a^{\otimes m},
\end{align*}
where $dz=dz_1\wedge\cdots\wedge dz_n$,
and  from Proposition \ref{prop: construct function 1}, we can choose a smooth strictly plurisubharmonic  function $\psi_\delta$ on $X$ such that
\begin{equation*}
\psi_{\delta}|_D = |z|^2  + n \log(|z-w|^2 + \delta^2),
\end{equation*}
where $0<\epsilon, \delta\ll 1$ are parameters.
From the multiple coarse $L^p$-estimate condition,
we obtain a smooth section $u_{\epsilon, \delta}$ of $E^{\otimes m}$-valued $(n,0)$-form on $ X$ such that $\dbar u_{\epsilon,\delta} = \alpha_\epsilon$ and
\begin{equation}
\int_{X} |u_{\epsilon,\delta}|^p _{h^{\otimes m}}e^{-\psi_\delta}dV_\omega\leq C_m\int_{X}\langle B^{-1}_{\psi_{\delta} } \alpha_\epsilon,\alpha_\epsilon\rangle^{\frac{p}{2}} e^{-\psi_\delta}dV_\omega.\label{eqn:dbar-est}
\end{equation}

On $D$, we have the following estimate:
\begin{equation*}
\begin{split}
\langle B^{-1}_{\psi_{\delta} } \alpha_\epsilon,\alpha_\epsilon\rangle
& = | \chi' (\frac{|z-w|^2}{\epsilon^2} ) |^2 \cdot \frac{1}{\epsilon^4}
\langle B^{-1}_{\psi_{\delta} }\sum_j (z_j-w_j) d\overline{z}_j\wedge dz\otimes a^{\otimes m},\sum_j (z_j-w_j) d\overline{z}_j\wedge dz\otimes a^{\otimes m}\rangle\\
& \leq C_1| \chi' (\frac{|z-w|^2}{\epsilon^2} ) |^2 \cdot \frac{1}{\epsilon^4}|z-w|^2|a|_{h(z)}^{2m},
\end{split}
\end{equation*}
where $C_1$ depends only on $\omega$.
Note that
$$\text{supp}\ \chi'\left(\frac{|z-w|^2}{\epsilon^2} \right)\subset\{1/4 \leq |z-w|^2 /\epsilon^2 \leq 1 \}$$
and $\psi_{\delta}\geq  2n\log|z-w|$, we have
\begin{equation}\label{ineq t3}
\begin{split}
(\text{RHS of (\ref{eqn:dbar-est})}) &\leq C_m C_1^{\frac{p}{2}}\int_{\{\epsilon^2 / 4 \leq |z-w|^2 \leq \epsilon^2\}}| \chi' (\frac{|z-w|^2}{\epsilon^2} )|^p \frac{1}{\epsilon^{2p}}|z-w|^p e^{- \psi_{\delta}}|a|_{h(z)}^{mp}dV_\omega\\
& \leq C_m C_1^{\frac{p}{2}}\frac{2^p}{\epsilon^{2p}} \int_{\{\epsilon^2 / 4 \leq |z-w|^2 \leq \epsilon^2\}} |z-w|^pe^{- \psi_{\delta}}|a|_{h(z)}^{mp}dV_\omega\\
&\leq C_m C_1^{\frac{p}{2}}\frac{2^p}{\epsilon^{2p}}\int_{\{\epsilon^2 / 4 \leq |z-w|^2 \leq \epsilon^2\}} \epsilon^p
\sup_{B(w,\epsilon)}|a|_{h(z)}^{mp} e^{-2n\log|z-w|}dV_\omega\\
&\leq C_2C_m \frac{\sup_{B(w,\epsilon)}|a|_{h(z)}^{mp}}{\epsilon^p},
\end{split}
\end{equation}
where $C_2=2^{p+2n}C_1^{\frac{p}{2}}C^n\mu(B_1)$ and $\mu(B_1)$ is the volume of the unit ball $B_1$ with respect to the Lebesgue measure.

To summarize, we have obtained a smooth section $u_{\epsilon,\delta}$ of $E^{\otimes m}$-valued $(n,0)$-form on $ X$ such that
\begin{itemize}
	\item $\dbar u_{\epsilon,\delta} = \alpha_\epsilon$, and
	\item the following estimate holds:
	\begin{equation}
	\int_{D} |u_{\epsilon,\delta}|_{h^{\otimes m}}^p e^{-\psi_{\delta}}dV_\omega \leq C_2C_m\frac{\sup_{B(w,\epsilon)}|a|_{h(z)}^{mp}} {\epsilon^p}.\label{eqn:matome}
	\end{equation}
\end{itemize}

Note that the weight function $\psi_{\delta}$ is decreasing when $\delta \searrow 0$,  $e^{-\psi_{\delta}}$ is increasing when $\delta \searrow 0$.
Fix $\delta_0>0$. Then, for $\delta < \delta_0$,   we have that
$$\int_{D} |u_{\epsilon,\delta}|_{h^{\otimes m}}^p e^{-\psi_{\delta_0}}dV_\omega \leq \int_{D} |u_{\epsilon,\delta}|_{h^{\otimes m}}^p e^{- \psi_{\delta}}dV_\omega \leq C_2C_m\frac{\sup_{B(w,\epsilon)}|a|_{h(z)}^{mp}} {\epsilon^p}.$$
Thus $\{u_{\epsilon,\delta}\}_{\delta < \delta_0}$ forms a bounded sequence in
$L^p(X, K_X\otimes E^{\otimes m}, e^{-\psi_{\delta_0}})$.
Note that $p>1$, we can choose a sequence $\{u_{\epsilon,\delta^{(k)}}\}_k$ in $L^p(X, e^{-\delta_0})$
which weakly converges to some $u_\epsilon\in L^p(X, K_X\otimes E^{\otimes m}, e^{-\psi_{\delta_0}})$, satisfying
$$\int_{D} |u_\epsilon|_{h^{\otimes m}}^p e^{-\psi_{\delta_0}} dV_\omega\leq  C_2C_m\frac{\sup_{B(w,\epsilon)}|a|_{h(z)}^{mp}} {\epsilon^p}. $$
Repeating this argument for a sequence $\{\delta_j\}$ decreasing to $0$,
by diagonal argument,
we can select a sequence $\{u_{\epsilon,\delta^k}\}_k$ which weakly converges to $u_\epsilon$
in $L^p(X, K_X\otimes E^{\otimes m}, e^{-\psi_{\delta_j}})$ with $u_\epsilon$ satisfying
$$\int_{D} |u_\epsilon|_{h^{\otimes m}}^p e^{-\psi_{\delta_j}}dV_\omega \leq C_2C_m\frac{\sup_{B(w,\epsilon)}|a|_{h(z)}^{mp}} {\epsilon^p}$$
for all $j$.
By the monotone convergence theorem,
$$\int_{  D} |u_\epsilon|_{h^{\otimes m}}^p e^{- \psi_{0}}dV_\omega \leq  C_2C_m\frac{\sup_{B(w,\epsilon)}|a|_{h(z)}^{mp}} {\epsilon^p}. $$
Since $\bar\partial$ is weakly continuous, we also have $\dbar u_\epsilon = \alpha_\epsilon.$

Since $\frac{1}{|z-w|^{2n}}$ is not integrable near $w$, $u_\epsilon(w)$ must be $0$.
Let
$$f_\epsilon:= \chi(|z-w|^2/\epsilon^2)dz\otimes a^{\otimes m} - u_\epsilon.$$
Then $f_\epsilon\in H^0(X,\wedge^{(n,0)}T^*_X\otimes E^{\otimes m})$,
$f_\epsilon(0) = dz \otimes a^{\otimes m}$ and
\begin{equation}\label{eqn:4}
\begin{split}
\int_{D} |f_\epsilon|_{h^{\otimes m}}^pdV_\omega
&\leq \left(\left(\int_{D} \left|\chi(|z-w|^2/\epsilon^2)dz\otimes a^{\otimes m}\right|_{h^{\otimes m}}^pdV_\omega\right)^{1/p}
+ \left(\int_{D} |u_\epsilon|_{h^{\otimes m}}^pdV_\omega\right)^{1/p}\right)^p\\
&\leq 2^p\left(\int_{D} \left|\chi(|z-w|^2/\epsilon^2)dz\otimes a^{\otimes m}\right|_{h^{\otimes m}}^pdV_\omega +\int_{D} |u_\epsilon|_{h^{\otimes m}}^pdV_\omega \right)
\end{split}
\end{equation}

Since $\chi \leq 1$ and the support of $\chi(|z-w|^2 / \epsilon^2) $ is contained in $\{|z-w|^2 \leq \epsilon^2 \}$ and $0<\epsilon\leq1$, we have
$$\int_{D} \left|\chi(|z-w|^2/\epsilon^2)dz\otimes a^{\otimes m}\right|_{h^{\otimes m}}^pdV_\omega\leq C^n\mu( B_1)\sup_{B(w,\epsilon)}|a|_{h(z)}^{mp}.$$
We also have
\begin{align*}
\int_{D}|u_\epsilon|_{h^{\otimes m}}^pdV_\omega &\leq \sup_{z \in D} e^{\psi_0(z)} \cdot \int_{D}
|u_\epsilon|_{h^{\otimes m}}^p e^{-\psi_0 }dV_\omega \\
&\leq\sup_{z \in D}e^{\psi_0(z)}  \cdot C_2C_m\frac{\sup_{B(w,\epsilon)}|a|_{h(z)}^{mp}}{\epsilon^p} \\
&\leq C_3C_m\frac{\sup_{B(w,\epsilon)}|a|_{h(z)}^{mp} }{\epsilon^p},
\end{align*}
where $C_3$ is a constant depends only on $D$. We may assume $C_m\geq1$.
Combining these estimates with \eqref{eqn:4}, we obtain that
\begin{equation*}
\int_{D} |f_\epsilon|_{h^{\otimes m}}^p dV_\omega \leq C_4 C_m\frac{\sup_{B(w,\epsilon)}|a|_{h(z)}^{mp}}{\epsilon^p},
\end{equation*}
where $C_4$ is a constant independent of $m$ and $w$.

Let
$$O_\epsilon=\sup\limits_{z,w\in D, |z-w|\leq\epsilon} \left|\log |a|_{h(z)}-\log |a|_{h(w)}\right|.$$
By the uniform continuity of $\log |a|_{h(z)}$ on $D$, $O_\epsilon$ is finite and goes to 0 as $\epsilon\ra 0$.
Let $\epsilon := 1/m$. We have  $\left|mp\log |a|_{h(z)} - mp\log |a|_{h(w)}\right| \leq mp O_{1/m}$ for $|z-w| \leq 1/m$. Then
\begin{align}
\int_{D} |f_{1/m}|^p e^{-m\phi}dV_\omega
&\leq C_4C_m m^{p} e^{mp\log |a|_{h(w)}+mp O_{1/m}}\nonumber\\
&= C_4C_mm^{p} e^{mp O_{1/m}}.\label{eqn:last}
\end{align}
Let $C'_m=C_4C_mm^{p} e^{mp O_{1/m}}$, we have
$$\frac{\log C'_m}{m}=\frac{\log (C''C_mm^{p})}{m}+ pO_{1/m}\to 0.$$
Considering $f_{1/m}/dz$, we see that $(E,h)$ satisfies the multiple coarse $L^p$-extension condition on $D$,
and hence $(E,h)$ is Griffiths semi-positive on $D$ by \cite[Theorem 1.2]{DWZZ1}.
Since $D$ is arbitrary, $(E,h)$ is Griffiths semi-positive on $X$.
\end{proof}

\subsection{Griffiths positivity in terms of optimal $L^p$-extension condition}

\begin{thm}[=Theorem \ref{thm: optimal Lp extension : Griffiths positive-intr}]\label{thm: optimal Lp estimate : Griffiths positive}
Let $E$ be a holomorphic vector bundle over a domain $D\subset\mc^n$,
and $h$ be a singular Finsler metric on $E$,
 such that $|s|_{h^*}$ is upper semi-continuous
for any local holomorphic section $s$ of $E^*$.
If $(E,h)$ satisfies the optimal $L^p$-extension condition for some $p>0$,
then $(E,h)$ is Griffiths semi-positive.
\end{thm}

\begin{proof}
Let $u$ be a holomorphic section of $E^*$ over $D$.
Let $z\in D$ and $P$ be any holomorphic cylinder such that $z+P\subset D$.
Take $a\in E_z$ such that $|a|_h=1$ and  $|u|_{h^*}(z)=|\langle u(z),a\rangle|$.
Since $(E,h)$ satisfies the optimal $L^p$-extension condition,
there is a holomorphic section $f$ of $E$ on $z+P$,
such that $f(z)=a$ and satisfies the estimate
	\begin{align}\label{eqn: optimal Lp extension a}
	\frac{1}{\mu(P)}\int_{z+P}|f|^p_h\leq 1.
	\end{align}
	Note that $|u|_{h^*}\geq |\langle u,f\rangle|/|f|_h$ on $z+P$,
it follows that $$\log|u|_{h^*}\geq \log|\langle u,f\rangle|-\log|f|_h.$$
	Taking integration,  we get that
	\begin{align*}
	p\left(\frac{1}{\mu(P)}\int_{z+P}\log |u|_{h^*}\right)
    &\geq p\left(\frac{1}{\mu(P)}\int_{z+P}\log|\langle u,f\rangle|\right)- \frac{1}{\mu(P)}\int_{z+P}\log|f|^p_h\\
	&\geq p\left(\frac{1}{\mu(P)}\int_{z+P}\log|\langle u,f\rangle|\right)-\log\left(\frac{1}{\mu(P)}\int_{z+P}|f|^p_h\right)\\
	&\geq p\log|\langle u(z),f(z)\rangle|\\
	&=p\log|\langle u(z), a\rangle|=p\log|u(z)|_{h^*},
	\end{align*}
	where the second inequality follows from Jensen's inequality and \eqref{eqn: optimal Lp extension a},
and the third inequality follows from the fact that $\log|\langle u,f\rangle|$ is a plurisubharmonic function,
and from \cite[Lemma 3.1]{DNW1}. Dividing by $p$, we obtain that
	\begin{align*}
	\log|u(z)|_{h^*}\leq \frac{1}{\mu(P)}\int_{z+P}\log |u|_{h^*}.
	\end{align*}
Again from \cite[Lemma 3.1]{DNW1}, we see that $\log |u|_{h^*}$ is plurisubharmonic on $D$.
\end{proof}

\subsection{Griffiths positivity in terms of multiple coarse $L^p$-extension condition}

 The following theorem was originally given in \cite[Theorem 6.4]{DWZZ1}.
 In the present paper, we give a new proof of it based on Guan-Zhou's idea \cite{GZh15d} about
 connecting optimal $L^2$-extension condition to Berndtsson's plurisubharmonic variation of relative Bergman kernels \cite{Bob06}.

\begin{thm}[=Theorem\ref{thm: multiple coarse Lp: Griffiths positivity-intr.}]\label{thm: multiple coarse Lp: Griffiths positivity.}
Let $E$ be a holomorphic vector bundle over a domain $D\subset\mc^n$,
and $h$ be a singular Finsler metric on $E$,
 such that $|s|_{h^*}$ is upper semi-continuous
for any local holomorphic section $s$ of $E^*$.
If $(E,h)$ satisfies the multiple coarse  $L^p$-extension condition, then $(E,h)$ is Griffiths semi-positive.
	\end{thm}
\begin{proof}
Let $u$ be a holomorphic section of $E^*$ over $D$.
Then $u^{\otimes m}\in H^0(D,({E^*})^{\otimes m})$.
	
Let $z\in D$ and $P$ be any holomorphic cylinder such that $z+P\subset D$. Take $a\in E_z$ such that $|a|_h=1$ and  $|u|_{h^*}(z)=|\langle u(z),a\rangle|$. Since $(E,h)$ satisfies the multiple coarse  $L^p$-extension condition,  there is $f_m\in H^0(D, E^{\otimes m})$, such that $f_m(z)=a^{\otimes m}$ and satisfies the following estimate
		$$\int_D|f_m|^p\leq C_m,$$
		where $C_m$ are constants independent of $z$ and satisfy the growth condition $\frac{1}{m}\log C_m\rightarrow 0$ as $m\rightarrow \infty$.
		Since $|u^{\otimes m}|_{(h^*)^{\otimes m}}=|u|^m_{h^*}\geq \frac{|\langle u^{\otimes m},f_m\rangle|}{|f_m|_{h^{\otimes m}}}$, we have that $$m\log |u|_{h^*}\geq \log|\langle u^{\otimes m},f_m\rangle|-\log |f_m|.$$
		Taking integration, we get that
		\begin{align*}
		m\left(\frac{1}{\mu(P)}\int_{z+P}\log |u|_{h^*}\right)&\geq  \frac{1}{\mu(P)}\int_{z+P}\log |\langle u^{\otimes m},f_m\rangle|-\frac{1}{p}\left(\frac{1}{\mu(P)}\int_{z+P}\log |f_m|^p\right)\\
		&\geq m\log|u(z)|_{h^*}-\frac{1}{p}\log\left(\frac{1}{\mu(P)}\int_{z+P}|f_m|^p\right)\\
		&\geq m\log|u(z)|_{h^*}-\frac{1}{p}\log\left(\frac{1}{\mu(P)}\int_{D}|f_m|^p\right)\\
		&\geq m\log|u(z)|_{h^*}-\frac{1}{p}\log(C_m/\mu(P)),
		\end{align*}
		where the first inequality follows from the fact that $\log |\langle u^{\otimes m},f_m\rangle|$ is a plurisubharmonic function, and \cite[Lemma 3.1]{DNW1}, and Jensen's inequality, and the second inequality follows from the fact that $z+P\subset D$.
 Dividing by $m$ in both sides, we obtain that
 \begin{align*}
 \frac{1}{\mu(P)}\int_{z+P}\log |u|_{h^*}\geq \log|u(z)|_{h^*}-\frac{1}{mp}\log(C_m/\mu(P)).
 \end{align*}
		Letting $m\rightarrow \infty$, we see that $\log|u|_{h^*}$ satisfies the following inequality
		$$ \log|u(z)|_{h^*}\leq \frac{1}{\mu(P)}\int_{z+P}\log |u|_{h^*},$$
		since $\frac{1}{m}\log C_m\rightarrow 0$ as $m\rightarrow \infty$.  Then from \cite[Lemma 3.1]{DNW1}, we get that $\log|u|_{h^*}$ is plurisubharmonic on $D$.
 	\end{proof}

%

\section{Positivities of direct images of twisted relative canonical bundles}
\subsection{Optimal $L^2$-estimate condition and Nakano positivity}
The aim of this subsection is to prove Theorem \ref{thm: direc im stein-intr} and Theorem \ref{thm: direct image-optimal L2 estimate-intr}.

To avoid some complicated geometric quantities and highlight the main idea,
we first consider a simple case of Theorem \ref{thm: direc im stein-intr} as a warm-up.

\begin{thm}\label{thm direc im stein}
Let $U$ and $D$ be bounded domains in $\mathbb{C}_t^{n}$ and $\mathbb{C}_z^{m}$ respectively, and
$\phi\in\mathcal{C}^2(\overline{U}\times \overline{D})\cap PSH(U\times D)$.
 Assume that $D$ is pseudoconvex.
For $t\in U$, let $A^2_t:=\{f\in\mathcal{O}(D):||f||^2_t:=\int_{D}|f|^2e^{-\phi(t,\cdot)}<\infty\}$ and $F:=\coprod_{t\in U}A^2_t$.
We may view $F$ as a Hermitian holomorphic vector bundle on $U$. Then $(F,||\cdot||_t)$ is Nakano semi-positive.
\end{thm}
\begin{proof}
We will first prove that $(F,||\cdot||_t)$ satisfies the $\bar\partial$ optimal $L^2$-estimate for pseudoconvex domains contained in $U$.
We may assume $U$ is pseudoconvex.

For any smooth strictly plurisubharmonic function $\psi$ on $U$, for any $\bar\partial$ closed
$f\in\mathcal{C}^\infty_c(T^*_{U}\Lambda^{(0,1)}\otimes F)$ (We identify $\mathcal{C}^\infty_c(T^*_{U}\Lambda^{(0,1)}\otimes F)$ with
$\mathcal{C}^\infty_c(T^*_{U}\Lambda^{(n,1)}\otimes F)$).
We may write $f=\sum_{j=1}^{n}f_j(t,z)d\bar t_j$ with $f_j(t,\cdot)\in F_t$ for $t\in U$ and  $j=1,2,\cdots,n$.
Therefore, we may view $f$ as a $\bar\partial$-closed $(0,1)$-form on $U\times D$.
By Lemma \ref{thm: L2 estimate Nakano}, there exists a function $u$ on $U\times D$,
satisfying $\bar\partial u=f$ and
\begin{equation*}
\begin{split}
&\int_{U\times D}|u|^2e^{-(\phi+\psi)}\\
\leq&\int_{U\times D}|f|^2_{i\partial\bar\partial(\phi+\psi)}e^{-(\phi+\psi)}\\
\leq&\int_{U\times D}|f|^2_{i\partial\bar\partial\psi}e^{-(\phi+\psi)}\\
=&\int_{U}\sum_{j,k=1}^{n}\psi^{j\bar k}\langle f_j(t,\cdot),f_k(t,\cdot)\rangle_te^{-\psi},
\end{split}
\end{equation*}
where $(\psi^{j\bar k})_{n\times n}:=(\frac{\partial^2\psi}{\partial t_j\partial_{\bar t_k}})^{-1}_{n\times n}$.
Note that  $\int_{U\times D}|u|^2e^{-(\phi+\psi)}=\int_{U}||u||_t^2e^{-\psi}<\infty$ and $\frac{\partial u}{\partial\bar z_j}=0$
for $j=1,2,\cdots,m$, we may view $u$ as a $L^2$-section of $F$ on $U$.
By Theorem \ref{thm:theta-nakano text_intr} and Remark \ref{rem:reduce to trivial bundle},
$(F,||\cdot||_t)$ is Nakano semi-positive.
\end{proof}

Let $\Omega=U\times D\subset \mc^n_t\times\mc^m_z$ be a bounded pseudoconvex domains and $p:\Omega\ra U$ be the natural projection.
Let $h$ be a Hermitian metric on the trivial bundle $E=\Omega\times\mc^r$ that is $C^2$-smooth to $\overline\Omega$.
For $t\in U$, let
$$F_t:=\{f\in H^0(D, E|_{\{t\}\times D}):\|f\|^2_{t}:=\int_D|f|^2_{h_t}<\infty\}$$
and $F:=\coprod_{t\in U}F_t$.
Since $h$ is continuous to $\overline\Omega$, $F_t$ are equal for all $t\in U$ as vector spaces.
We may view $(F, \|\cdot\|)$ as a trivial holomorphic Hermitian vector bundle of infinite rank  over $U$.

\begin{thm}(=Theorem \ref{thm: direc im stein-intr})\label{thm: direc im stein}
Let $\theta$ be a continuous real $(1,1)$-form on $U$ such that $i\Theta_{E}\geq p^*\theta\otimes Id_E$,
then $i\Theta_{F}\geq \theta\otimes Id_F$ in the sense of Nakano.
In particular, if $i\Theta_{E}>0$ in the sense of Nakano, then $i\Theta_{F}>0$ in the sense of Nakano.
\end{thm}
\begin{proof}
	By Theorem \ref{thm:theta-nakano text}, it suffices to prove that $(F,\|\cdot\|)$ satisfies:
for any $f\in\mathcal{C}^\infty_c(U,\wedge^{n,1}T^*_U\otimes F\otimes A)$ with $\bar\partial f=0$,
and any positive Hermitian line bundle $(A,h_A)$ on $U$ with $i\Theta_{A,h_A}+\theta>0$ on $\text{supp}f$,
there is $u\in L^2(U,\wedge^{n,0}T_U^*\otimes F\otimes A)$, satisfying $\bar\partial u=f$ and
$$\int_U|u|^2_{h\otimes h_A}dV_\omega\leq \int_U\langle B_{h_A,\theta}^{-1}f,f\rangle_{h\otimes h_A} dV_\omega,$$
provided that the right hand side is finite,
where $B_{h_A,\theta}=[(i\Theta_{A,h_A}+\theta)\otimes Id_F,\Lambda_\omega]$.
	
 We may write $f=\sum_{j=1}^{n}f_j(t,z)dt\wedge d\bar t_j$ with $f_j(t,\cdot)\in F_t\otimes A$ for $t\in U$ and  $j=1,2,\cdots,n$.
 Therefore, we may view $f$ as a $\bar\partial$-closed $E\otimes p^*A$-valued $(n,1)$-form on $\Omega$. Let $\tilde{f}=f\wedge dz$,
then $\tilde{f}$ is a $\bar\partial$-closed $E\otimes p^*A$-valued $(m+n,1)$-form on $\Omega$.
 By assumption, $i\Theta_{E}\geq p^*\theta\otimes Id_E$. We get
$$i\Theta_{E}+ip^*(\Theta_{A,h_A})\otimes Id_E\geq p^*(\theta+i\Theta_{A,h_A})\otimes Id_E.$$
Therefore,
\begin{equation*}
\begin{split}
&\langle[i\Theta_{E}+ip^*(\Theta_{A,h_A})\otimes Id_E,\Lambda_\omega]^{-1}\tilde{f},\tilde{f}\rangle_{h\otimes h_A}\\
\leq&\langle[p^*(\theta+i\Theta_{A,h_A})\otimes Id_E,\Lambda_\omega]^{-1}\tilde{f},\tilde{f}\rangle_{h\otimes h_A}
\end{split}
\end{equation*}


	By Lemma \ref{thm: L2 estimate Nakano}, we can find an $E\otimes p^*A$-valued $(n+m,0)$-form $\tilde u$ on $\Omega$,
	satisfying $\bar\partial\tilde u= \tilde f$ and
	\begin{equation*}
	\begin{split}
	&\int_{\Omega}|\tilde u|_{h\otimes h_A}^2\\
	\leq&\int_{\Omega}\langle[p^*(\theta+i\Theta_{A,h_A})\otimes Id_E,\Lambda_\omega]^{-1}\tilde{f},\tilde{f}\rangle_{h\otimes h_A}\\
	=&\int_U\langle B_{h_A,\theta}^{-1}f,f\rangle_{h\otimes h_A},
	\end{split}
	\end{equation*}
where the last equality holds by the Fubini theorem.
 Since $\frac{\partial\tilde u}{\partial z_j}=0$, $\tilde u$ is holomorphic along fibers and we may view $u=\tilde u/dz$ as a section of $K_U\otimes F\otimes A$.
Also by the Fubini theorem, we have
$$\int_{\Omega}|\tilde u|^2_{h\otimes h_A}=\int_{U}||u||_{h\otimes h_A}^2<\infty.$$
We also have $\bar\partial u=f$.
Hence $(F,\|\cdot\|)$ satisfies the optimal $L^2$-estimate condition and is Nakano semi-positive
by \ref{thm:theta-nakano text}.
\end{proof}

Let $\pi:X\rightarrow U$ be a proper holomorphic submersion from  K\" ahler manifold $X$ of complex dimension $m+n$,  to a  bounded pseudoconvex domain  $U\subset \mathbb C^n$, and $(E,h)$ be a Hermitian holomorphic vector bundle over $X$, with the Chern curvature Nakano semi-positive.
From Lemma \ref{thm: optimal L2 extension}, the direct image $F:=\pi_*(K_{X/U}\otimes E)$ is a vector bundle,
whose fiber over $t\in U$ is $F_t=H^0(X_t, K_{X_t}\otimes E|_{X_t})$.
There is a hermtian metric $\|\cdot\|$ on $F$ induced by $h$:
for any $u\in F_t$,
$$\|u(t)\|^2_t:=\int_{X_t}c_mu\wedge \bar u,$$
where $m=\dim X_t$, $c_m=i^{m^2}$, and $u\wedge \bar u$ is the composition of the wedge product and the inner product on $E$.
So we get a Hermitian holomorphic vector bundle $(F, \|\cdot\|)$ over $U$.

\begin{thm}(=Theorem \ref{thm: direct image-optimal L2 estimate-intr})\label{thm: direct image-optimal L2 estimate}
The Hermitian holomorphic vector bundle $(F, \|\cdot\|)$ over $U$ defined above
satisfies the optimal $L^2$-estimate condition.
Moreover, if $i\Theta_{E}\geq p^*\theta\otimes Id_E$ for a continuous real $(1,1)$-form $\theta$ on $U$,
then $i\Theta_{F}\geq \theta\otimes Id_F$ in the sense of Nakano.
\end{thm}
\begin{proof}
Similar to the proof of Theorem \ref{thm: direc im stein}, we may assume $\theta=0$.
From Theorem \ref{thm:theta-nakano text}, it suffices to prove that $(\pi_*(K_{X/Y}\otimes E), \|\cdot\|)$
satisfies the optimal $L^2$-estimate condition with the standard K\"ahler metric $\omega_0$ on $U\subset\mc^n$.
Let $\omega$ be an arbitrary K\"ahler metric on $X$.

Let $f$ be a $\bar\partial$-closed  compact supported smooth $(n,1)$-form with values in $F$,
and let $\psi$ be any smooth strictly plurisubharmonic function on $U$.

We can write $f(t)=dt\wedge(f_1(t)d\bar t_1+\cdots +f_n(t)d\bar t_n)$,
with $f_i(t)\in F_t=H^0(X_t, K_{X_t}\otimes E)$.
One can identify  $f$ as a smooth compact supported $(n+m,1)$-form $\tilde f(t,z):=dt\wedge (f_1(t,z)d\bar t_1+\cdots+f_n(t,z)d\bar t_n)$ on $X$,
with $f_i(t,z)$ being holomorphic section of $K_{X_t}\otimes E|_{X_t}$.
We have the following observations:
\begin{itemize}
\item $\bar\partial_zf_i(t,z)=0$ for any fixed $t\in B$, since  $f_i(t,z)$ are holomorphic  sections $K_{X_t}\otimes E|_{X_t}$.
\item $\bar\partial_tf=0$, since $f$ is a $\bar\partial$-closed form on $B$.
\end{itemize}
It follows that $\tilde f$ is a $\bar\partial$-closed compact supported  $(n+m,1)$-form on $X$ with values in $E$.
We want to solve the equation $\bar\partial u=\tilde f$ on $X$ by using Lemma \ref{thm: L2 estimate Nakano}.
Now we equipped $E$ with the metric $\tilde h:=he^{-\pi^*\psi}$,
then $i\Theta_{E,\tilde h}=i\Theta_{E,h}+i\partial\bar\partial\pi^*\psi\otimes Id_{E}$, which is also  semi-positive in the sense of Nakano.

We consider the integration
$$\int_X\langle [i\Theta_{E,h}+i\partial\bar\partial \pi^*\psi\otimes Id_E, \Lambda_{\omega}]^{-1}\tilde f,\tilde f\rangle e^{-\pi^*\psi}dV_\omega.$$
Note that, acting on $\Lambda^{n+m,1}T^*_X\otimes E$, by Lemma \ref{lem: Hermitian operator formula}, we have
\begin{align*}
[i\Theta_{E,h}+i\partial\bar\partial \pi^*\psi\otimes Id_E, \Lambda_{\omega}]
\geq [i\partial\bar\partial \pi^*\psi\otimes Id_E, \Lambda_{\omega}].
\end{align*}
Thus we obtain that
\begin{align*}
[i\Theta_{E,h}+i\partial\bar\partial \pi^*\psi\otimes Id_E, \Lambda_{\omega}]^{-1}\leq  [i\partial\bar\partial \pi^*\psi\otimes Id_E, \Lambda_{\omega}]^{-1}.
\end{align*}

For any $p\in X$, we use Lemma \ref{lem:inverse cur. indep of metric} to modify $\omega$ at $p$.
We take a local coordinate $(t_1, \cdots, t_n, z_1, \cdots, c_m)$ on $X$ near $p$,
where $t_1,\cdots, t_n$ is the standard coordinate on $U\subset\mc^n$.
Let $\omega'=i\sum_{j=1}^n dt_j\wedge d\bar{t}_j+i\sum_{l=1}^{m}dz_l\wedge d\bar{z}_l$.

Note that
$$i\partial\bar\partial \pi^*\psi=\sum_{j=1}^{n}\frac{\partial^2\psi}{\partial t_j\partial\bar{t}_k}dt_j\wedge d\bar{t}_k,$$
we have
$$[i\partial\bar\partial \pi^*\psi\otimes Id_E, \Lambda_{\omega'}]\tilde f
=\sum_{j,k}\frac{\partial^2\psi}{\partial t_j\partial\bar{t}_k}f_j(t,z)dt\wedge d\bar{t}_k,$$
and
$$[i\partial\bar\partial \pi^*\psi\otimes Id_E, \Lambda_{\omega'}]^{-1}\tilde f
=\sum_{j,k}\psi^{jk}f_j(t,z)dt\wedge d\bar{t}_k$$
at $p$, where $(\psi^{jk})=(\frac{\partial^2\psi}{\partial t_j\partial\bar{t}_k})^{-1}.$
By Lemma \ref{lem:inverse cur. indep of metric}, we have
\begin{equation*}
\begin{split}
 &\langle[i\partial\bar\partial \pi^*\psi\otimes Id_E, \Lambda_{\omega}]^{-1}\tilde f,\tilde f\rangle_{\omega}dV_{\omega}\\
=&\langle[i\partial\bar\partial \pi^*\psi\otimes Id_E, \Lambda_{\omega'}]^{-1}\tilde f,\tilde f\rangle_{\omega'}dV_{\omega'}\\
=&\sum_{j,k}\psi^{jk}c_m f_j\wedge \bar{f}_kc_ndt\wedge d\bar{t}.
\end{split}
\end{equation*}

By Fubini's theorem, we get that
\begin{align*}
 &\int_X\langle  [i\partial\bar\partial \pi^*\psi\otimes Id_E, \Lambda_{\omega}]^{-1}\tilde f,\tilde f\rangle_\omega e^{-\pi^*\psi} dV_\omega\\
=&\int_X\sum_{j,k}\psi^{jk}c_m f_j\wedge \bar{f}_k e^{-\pi^*\psi}c_ndt\wedge d\bar{t}\\
=&\int_U<f_j,f_k>_t\psi^{jk}e^{-\psi}c_ndt\wedge d\bar t\\
=&\int_U\langle [i\partial\bar\partial \psi, \Lambda_{\omega_0}]^{-1}f,f\rangle_te^{-\psi}dV_{\omega_0},
\end{align*}
where by $\langle\cdot,\cdot\rangle_t$, we mean that pointwise inner product with respect to the Hermitian metric $\|\cdot\|$ of $F$.

From Lemma \ref{thm: L2 estimate Nakano}, there is  $\tilde u\in \Lambda^{m+n,0}(X,E)$, such that $\bar\partial\tilde u=\tilde f$, and satisfies the following estimate
\begin{align}\label{eqn: optimal L2 estimate 1}
     &\int_Xc_{m+n}\tilde u\wedge \bar{\tilde u}e^{-\pi^*\psi}\notag\\
\leq &\int_X\langle [i\Theta_{E,h}+i\partial\bar\partial \pi^*\psi\otimes Id_E, \Lambda_{\omega}]^{-1}\tilde f,\tilde f\rangle e^{-\pi^*\psi}dV_\omega\\
\leq &\int_U\langle [i\partial\bar\partial \psi, \Lambda_{\omega_0}]^{-1}f,f\rangle_te^{-\psi}dV_{\omega_0}.\notag
\end{align}

 We observe that $\bar\partial\tilde u|_{X_t}=0$ for any fixed $t\in U$,
 since $\bar\partial\tilde u=\tilde f$.
 This means that $\tilde u_t:=\tilde u(t,\cdot)\in F_t$.
 Therefore we may view $\tilde u$ as a section $u$ of $F$.
 It is obviously that $\bar\partial u=f$.

From Fubini's theorem, we have that
\begin{align}\label{eqn:optimal L2 estimate 2}
\int_Xc_{m+n}\tilde u\wedge \bar{\tilde u}e^{-\pi^*\psi}=\int_U\|u\|_t^2e^{-\psi}dV_{\omega_0}.
\end{align}

Combining \eqref{eqn: optimal L2 estimate 1}, we have
\begin{align*}
\int_U\|u\|_t^2e^{-\psi}dV_{\omega_0}\leq \int_U\langle [i\partial\bar\partial \psi, \Lambda_{\omega_0}]^{-1}f,f\rangle_te^{-\psi}dV_{\omega_0}.
\end{align*}
We have proved that $F$ satisfies the optimal $L^2$-estimate condition, thus from Theorem \ref{thm:theta-nakano text} (and Remark \ref{rem:reduce to trivial bundle}),  $F$ is Nakano semi-positive.
\end{proof}

\subsection{Multiple coarse $L^2$-estimate condition and Griffiths positivity}
We apply Theorem \ref{thm:coarse estimate text-intr} and the fiber product technique introduced in \cite{DWZZ1} to provide a new method to study the Griffiths positivity of direct images.
\begin{thm}\label{thm: direct image-coarse L2 estimate}
The Hermitian holomorphic vector bundle $(F, \|\cdot\|)$ over $U$ as in Theorem \ref{thm: direct image-optimal L2 estimate}
satisfies the multiple coarse  $L^2$-estimate condition.
In particular, $F$ is Griffiths semipositive.
\end{thm}
\begin{proof}
Let $\omega_0$ be the standard K\" ahler metric on $U$ and $\omega$ be an arbitrary K\" ahler metric on $X$.
We have the following constructions:
\begin{itemize}
\item Let $X_k:=X\times_\pi\cdots\times_\pi X$ be the $k$ times fiber product of $X$ with respect to the map $\pi: X\rightarrow U$.

\item The induced map  $X_k\rightarrow U$  by $\pi$ is  denoted by $\pi_k: X_k\rightarrow U$,  and $X_{k,t}:=\pi^{-1}_k(t)=X_t^k$ for every $t\in U$.
\item There are natural holomorphic projections $pr_j$ from $X_k$ to its $j$-th factor $X$.
\item The induced  K\" ahler metric $\omega_k:=pr_1^*\omega+\cdots+pr_k^*\omega$ on $X_k$.
\item Set $E_j:=pr_j^*E$, and $E^k:=E_1\otimes \cdots \otimes E_k$. Then $E^k$ can be equipped with the induced metric $h^k:=pr_1^*h\otimes \cdots\otimes pr_k^*h$.
\end{itemize}

We have the following observations:
\begin{itemize}
\item From Lemma \ref{lem: Nakano fiber tensor product}, $E^k$ equipped with the Hermitian metric $h^k$ is Nakano semi-positive.
\item From \cite[Lemma 9.2]{DWZZ1},  the direct image bundle $F^k:=(\pi_k)_*(K_{X_k/U}\otimes E^k)=(\pi_*(K_{X/U}\otimes E))^{\otimes k}=F^{\otimes k}$, as Hermitian holomorphic vector bundles. (In fact, \cite[Lemma 9.2]{DWZZ1} was proved for line bundles, but it is clear that the proof also works for vector bundles.)
\end{itemize}

Let $f$ be an arbitrarily fixed smooth compactly supported $(m,1)$-form on $U$ with valued in $F^k$, such that $\bar\partial f=0$. Let $\psi$ be an arbitrary smooth strictly plurisubharmonic function on $U$.  To prove that  $F$ satisfies the multiple coarse $L^2$-estimate, we need to show that one can solve the equation $\bar\partial u=f$ on $U$, with the estimate $\int_U |u|^2_{h^k}e^{-\psi}\leq \int_U \langle  B_{\psi}^{-1}f,f  \rangle e^{-\psi}$, where $B_\psi=[i\partial\bar\partial \psi, \Lambda_{\omega_0}]$.

As in the proof of Theorem \ref{thm: direct image-optimal L2 estimate}, we may consider $f$ as a smooth compactly supported $K_{X_k}\otimes E^k$ valued $(0,1)$-form $\tilde f$ on $X_k$. Then it is clear that $\bar\partial \tilde f=0$.
We consider the following integration
$$\int_{X_k}\langle [i\Theta_{E^k,h^k}+i\partial\bar\partial \pi_k^*\psi\otimes Id_{E^k}, \Lambda_{\omega_k}]^{-1}\tilde f,\tilde f\rangle_{\omega_k} e^{-\pi_k^*\psi}dV_{\omega_k}.$$
 By the same  analysis as in the proof of Theorem \ref{thm: direct image-optimal L2 estimate}, we can get that
\begin{align*}
&\int_{X_k}\langle [i\Theta_{E^k,h^k}+i\partial\bar\partial \pi_k^*\psi\otimes Id_{E^k}, \Lambda_{\omega_k}]^{-1}\tilde f,\tilde f\rangle_{\omega_k} e^{-\pi_k^*\psi}dV_{\omega_k}\\
&\leq \int_{X_k}\langle   [i\partial\bar\partial \pi_k^*\psi\otimes Id_{E^k}, \Lambda_{\pi_k^*\omega_0}]^{-1}\tilde f, \tilde f    \rangle_{\omega_k} e^{-\pi_k^*\psi}dV_{\omega_k}\\
&=\int_{X_k}   \sum_{j,k}\psi^{jk}c_{km}f_j\wedge\bar f_k   e^{-\pi_k^*\psi} c_ndt\wedge d\bar t             \\
&=  \int_U\langle B_\psi^{-1}f,f\rangle_te^{-\psi}dV_{\omega_0}.
\end{align*}

Now from Lemma \ref{thm: L2 estimate Nakano}, we can solve the equation $\bar\partial\tilde u=\tilde f$ with the estimate
\begin{align*}
\int_{X_k}|\tilde u|^2_{h^k}e^{-\pi_k^*\psi}dV_{\omega_k}&\leq \int_{X_k}\langle [i\Theta_{E^k,h^k}+i\partial\bar\partial \pi_k^*\psi\otimes Id_{E^k}, \Lambda_{\omega_k}]^{-1}\tilde f,\tilde f\rangle_{\omega_k} e^{-\pi_k^*\psi}dV_{\omega_k}\\
&\leq \int_U\langle B_\psi^{-1}f,f\rangle_te^{-\psi}dV_{\omega_0}.
\end{align*}

Similarly,  $\bar\partial\tilde u|_{X_t}=0$ for any fixed $t\in U$,
since $\bar\partial\tilde u=\tilde f$.
This means that $\tilde u_t:=\tilde u(t,\cdot)\in F^k_t$.
Therefore we may view $\tilde u$ as a section $u$ of $F^k$.
It is obviously that $\bar\partial u=f$.


Applying Fubini's theorem to the L.H.S of above inequality, we get that
\begin{align*}
\int_U|u_t|_t^2e^{-\psi}dV_{\omega_0}\leq  \int_U\langle B_\psi^{-1}f,f\rangle_te^{-\psi}dV_{\omega_0},
\end{align*}
which implies that $(F, \|\cdot\|)$ satisfies the multiple coarse $L^2$-estimate on $U$.

\end{proof}

\subsection{Optimal $L^2$-extension condition and Griffiths positivity}
\begin{thm}\label{thm: direct image-optimal L2 extension}The Hermitian holomorphic vector bundle $(F, \|\cdot\|)$ over $U$ as in Theorem \ref{thm: direct image-optimal L2 estimate} satisfies the optimal $L^2$-extension condition. In particular, $F$ is Griffiths semipositive.
\end{thm}
\begin{proof}
For any $t_0\in U$, any holomorphic cylinder $P$ such that $t_0+P\subset U$, and any $a_{t_0}\in F_{t_0}$, which  is a holomorphic section of $K_{X_{t_0}}\otimes E|_{X_{t_0}}$ on $X_{t_0}$. Since $E$ is  Nakano semi-positive, from Lemma \ref{thm: L2 estimate Nakano}, we get a homolomorphic extension $ a\in H^0(X,K_X\otimes E)$ such that $ a|_{X_{t_0}}=a_{t_0}\wedge dt$, and with the estimate
\begin{align*}
\int_{\pi^{-1}(t_0+P)}c_{m+n} a\wedge \bar a\leq \mu(P)\int_{X_{t_0}}c_m a_{t_0}\wedge \bar a_{t_0}=\mu(P)|a_{t_0}|_{t_0}^2,
\end{align*}
where $\mu(P)$ is the volume of $P$ with respect to the Lebesgue measure $d\mu$ on $\mathbb C^m$.
Since $ a_t:=(a/dt)|_{X_t}\in H^0(X_t,K_{X_t}\otimes E|_{X_t})$,  $a/dt$ can be seen as a  holomorphic section of the direct image bundle $F $ over  $t_0+P$, and from Fubini's theorem, we can obtain that
\begin{align*}
\int_{t_0+P}|a_t|^2_tdV_{\omega_0}\leq \mu({P})|a_{t_0}|^2_{t_0},
\end{align*}
which is the desired optimal $L^2$-extension.

\end{proof}

\subsection{Multiple coarse $L^2$-extension condition and Griffiths positivity}
In this subsection, we will prove the following
\begin{thm} \label{thm: direct image-coarse L2 extension}
The Hermitian holomorphic vector bundle $(F, \|\cdot\|)$ over $U$ as in Theorem \ref{thm: direct image-optimal L2 estimate} satisfies the multiple coarse $L^2$-extension condition.  In particular, $F$ is Griffiths semipositive.
\end{thm}
\begin{proof} Let $(X_k, \pi_k, \omega_k, F^k)$ be as  in  the proof of Theorem \ref{thm: direct image-coarse L2 estimate}.
%
%
For any $t_0\in U$,  $a_{t_0}\in F_{t_0}$,  $a_{t_0}^{\otimes k}$ is a holomorphic section  of $K_{X_{k,t_0}}\otimes E^k$. Since $E^k$ with the induced metric $h^k$ is semi-positive in the sense of Nakano on $X_k$, by Lemma \ref{thm: optimal L2 extension}, there exists $ a\in H^0(X_k,K_{X_k}\otimes E^k)$, such that $ a|_{X_{k,t_0}}=a_{t_0}^{\otimes k}\wedge dt$ and satisfies  the following estimate
\begin{align*}
\int_{X_k}| a|^2_{h^k}dV_{\omega_k}\leq C |a_{t_0}^{\otimes k}|^2_{t_0},
\end{align*}
where $C$ is a universal constant which only depends on the diameter and dimension of  $U$.
We can view $ a_t:= (a/dt)|_{X_t}, t\in U$  as a holomorphic section of $F^k$. From Fubini's theorem,
we have that
\begin{align*}
\int_{X_k}| a_t|^2_{h^k}dV_{\omega_k}=\int_{U}| a_t|_t^2dV_{\omega_0}.
\end{align*}
In conclusion, we get a holomorphic extension $ a/dt$ of $a_{t_0}^{\otimes k}$ , with the estimate
\begin{align*}
\int_{U}| a_t|_t^2dV_{\omega_0}\leq C|a_{t_0}^{\otimes k}|^2_{t_0}.
\end{align*}
This completes the proof of Theorem \ref{thm: direct image-coarse L2 extension}.
\end{proof}

\begin{rem}\label{rem: Griffiths vector bundle map}  Let $\pi:X\rightarrow Y$ be a proper holomorphic map between K\"ahler manifolds which may be not regular. Let $(E,h)$ be a Hermitian holomorphic vector bundle on $X$ whose Chern curvature is Nakano semi-positive. Then the direct image sheaf $\mathcal F:=\pi_*(K_{X/Y}\otimes E)$ can be equipped with a natural singular metric which is positively curved in the sense of Definition \ref{def:finsler on sheaf}. In fact, let $Z\subset Y$ be the singular locus of $\pi$, then on $X\setminus \pi^{-1}(Z)$, $\pi$ is a submersion, and $\mathcal F$ is locally free
and can be viewed as a vector bundle $F$ on $Y':=Y\backslash Z$, with $F_t=H^0(K_{X_t}\otimes E|_{X_t})$. The induced Hermitian metric $\|\cdot\|$ on $F$ is as follows: for any holomorphic section $u\in H^0(Y', F)$,
	$$\|u\|_t^2:=\int_{X_t}c_{m}u\wedge\bar u.$$
	From one of Theorem \ref{thm: direct image-coarse L2 estimate},  and Theorem \ref{thm: direct image-optimal L2 extension},
Theorem \ref{thm: direct image-coarse L2 extension}, we see that $\|\cdot\|_t$ is a Hermitian metric on $F$ with Griffiths semi-positive curvature.
Moreover, by similar argument as in \cite[Propositionn 23.3]{HPS16} (see also \cite[Step 3 in the proof of Theorem 9.3]{DWZZ1}),
one can show that the metric on $F$ extends to a positively curved metric on  $\mathcal F$.
In the special case that $E$ is a line bundle, the same conclusion is true if $h$ is singular and pseudoeffective (see \cite{BP08, PT18, HPS16, DWZZ1,ZhouZhu}).
\end{rem}

\section*{Appendix}
We prove a result used in the proof of Theorem \ref{thm: direct image-optimal L2 estimate-intr},
which seems to be already known.

\begin{lem}\label{lem:inverse cur. indep of metric}
Let $U\subset\mathbb{C}^n$ be a domain, $\omega_1$, $\omega_2$ be any two Hermitian forms on $U$,
and $E=U\times \mathbb{C}^r$ be trivial vector bundle on $U$ with a Hermitian metric.
Let $\Theta\in C^0(X,\Lambda^{1,1}T^*_X\otimes End(E))$ such that $\Theta^*=-\Theta$.
Then
$$Im[i\Theta,\Lambda_{\omega_1}]=Im [i\Theta,\Lambda_{\omega_2}],$$
and for any $E$-valued $(n,1)$ form $u\in Im[i\Theta,\Lambda_{\omega_1}]$,
$$\langle[i\Theta,\Lambda_{\omega_1}]^{-1}u,u\rangle_{\omega_1}dV_{\omega_1}
=\langle[i\Theta,\Lambda_{\omega_2}]^{-1}u,u\rangle_{\omega_2}dV_{\omega_2}.$$
\end{lem}

\begin{proof}
For any $z_0\in U$, after a linearly transformation, we may assume
$\omega_1=i\sum_{j=1}^{n}dz_j\wedge d\bar{z}_j$ and $\omega_2=i\sum_{j=1}^{n}\lambda_j^2 dz_j\wedge d\bar{z}_j$
at $z_0$ with $\lambda_j>0$.
Let $w_j=\lambda_j z_j$ for $j=1,2,\cdots,n$,
 then $\omega_2=i\sum_{j=1}^{n}dw_j\wedge d\bar{w}_j$.
We may write
\begin{equation}\label{eq p1}
i\Theta=i\sum_{jk\alpha\beta}c_{jk\alpha\beta}dz_j\wedge d\bar{z}_k\otimes e^*_\alpha\otimes e_\beta
=i\sum_{jk\alpha\beta}c'_{jk\alpha\beta}dw_j\wedge d\bar{w}_k\otimes e^*_\alpha\otimes e_\beta
\end{equation}
with $c'_{jk\alpha\beta}=\frac{c_{jk\alpha\beta}}{\lambda_j\lambda_k}$.

Denote $\lambda=\prod_{j=1}^{n}\lambda_j$.
Let $u=\sum_{j,\alpha}u_{j\alpha}dz\wedge d\bar{z}_j\otimes e_\alpha$, then
$u=\sum_{j,\alpha}u'_{j\alpha}dw\wedge d\bar{w}_j\otimes e_\alpha$ with $u'_{j\alpha}=\frac{u_{j\alpha}}{\lambda\lambda_j}$.
Note that
\begin{equation}\label{eq p2}
[i\Theta,\Lambda_{\omega_1}]u=\sum_{jk\alpha\beta}u_{j\alpha}c_{jk\alpha\beta}dz\wedge d\bar{z}_k\otimes e_{\beta},
\end{equation}
and
\begin{equation}\label{eq p3}
[i\Theta,\Lambda_{\omega_2}]u=\sum_{jk\alpha\beta}u'_{j\alpha}c'_{jk\alpha\beta}dw\wedge d\bar{w}_k\otimes e_{\beta}.
\end{equation}
So it is easy to see $Im[i\Theta,\Lambda_{\omega_1}]=Im [i\Theta,\Lambda_{\omega_2}]$.
We write
$$[i\Theta,\Lambda_{\omega_1}]^{-1}u=\sum_{jk\alpha\beta}u_{j\alpha}d_{jk\alpha\beta}dz\wedge d\bar{z}_k\otimes e_{\beta},$$
$$[i\Theta,\Lambda_{\omega_2}]^{-1}u=\sum_{jk\alpha\beta}u'_{j\alpha}d'_{jk\alpha\beta}dw\wedge d\bar{w}_k\otimes e_{\beta},$$
Then from equations \eqref{eq p1},\eqref{eq p2},\eqref{eq p3}, we can get
$$d'_{jk\alpha\beta}=\lambda_j\lambda_k d_{jk\alpha\beta}.$$
We now assume that $\{e_\alpha\}$ are orthonormal at $z_0$.
Then
$$\langle[i\Theta,\Lambda_{\omega_1}]^{-1}u,u\rangle_{\omega_1}dV_{\omega_1}=
\sum_{jk\alpha\beta}d_{jk\alpha\beta}u_{j\alpha}\bar{u}_{k\beta}c_ndz\wedge d\bar{z},$$
$$\langle[i\Theta,\Lambda_{\omega_2}]^{-1}u,u\rangle_{\omega_2}dV_{\omega_2}=
\sum_{jk\alpha\beta}d'_{jk\alpha\beta}u'_{j\alpha}\bar{u'}_{k\beta}c_ndw\wedge d\bar{w}.$$
Note also that
 $$c_ndw\wedge d\bar{w}=\lambda^2c_ndz\wedge d\bar{z},$$
We get
$$\langle[i\Theta,\Lambda_{\omega_1}]^{-1}u,u\rangle_{\omega_1}dV_{\omega_1}=
\langle[i\Theta,\Lambda_{\omega_2}]^{-1}u,u\rangle_{\omega_2}dV_{\omega_2}.$$
\end{proof}


\begin{thebibliography}{10}
	

    \bibitem{Bob06}
    B.~Berndtsson.
    \newblock Subharmonicity conditions of the {B}ergman kernel and some other
    functions associated to pseudoconvex domains.
    \newblock {\em Ann. Inst. Fourier (Grenoble)}, 56(6):1633--1662, 2006.
	
	\bibitem{Bob09a}
	B.~Berndtsson.
	\newblock Curvature of vector bundles associated to holomorphic fibrations.
	\newblock {\em Ann. of Math. (2)}, 169(2):531--560, 2009.
	
		\bibitem{BP08}
	B.~Berndtsson and M.~P\u{a}un.
	\newblock Bergman kernels and the pseudoeffectivity of relative canonical
	bundles.
	\newblock {\em Duke Math. J.}, 145(2):341--378, 2008.
	
	\bibitem{Bob18}
	B.~Berndtsson.
	\newblock Complex Brunn-Minkowski theory and positivity of vector bundles.
	\newblock {\em arXiv:1807.05844}.
	
\bibitem{Ber-Pau10}	
B.~Berndtsson and M.~P\u{a}un. Bergman kernels and subadjunction, e-preprint, arXiv:1002.4145.

	\bibitem{Bl}
	Z.~B\l ocki.
	\newblock Suita conjecture and the {O}hsawa-{T}akegoshi extension theorem.
	\newblock {\em Invent. Math.}, 193(1):149--158, 2013.

\bibitem{Cao141}
J.~Cao.
\newblock Ohsawa-{T}akegoshi extension theorem for compact {K}\"ahler manifolds
and applications.
\newblock In {\em Complex and symplectic geometry}, volume~21 of {\em Springer
	INdAM Ser.}, pages 19--38. Springer, Cham, 2017.





	\bibitem{Dem}
	 J.-P. Demailly,
	 \newblock{Complex analytic and differential geometry},
	 \newblock{ http://www-fourier.ujf-grenoble.fr/$\sim$demailly/manuscripts/agbook.pdf.}
	
	\bibitem{Dem82}
	J.-P. Demailly.
	\newblock Estimations {$L^{2}$}\ pour l'op\'erateur {$\bar \partial $}\ d'un
	fibr\'e vectoriel holomorphe semi-positif au-dessus d'une vari\'et\'e
	k\"ahl\'erienne compl\`ete.
	\newblock {\em Ann. Sci. \'Ecole Norm. Sup. (4)}, 15(3):457--511, 1982.
	
	
	
%
	
%
%


	\bibitem{DNW1}
F. ~Deng, J.~ Ning and Z. ~ Wang.
\newblock Characterizations of plurisubharmonic functions.
\newblock{\em arXiv: 1910.06518v1}

	\bibitem{DWZZ1}
	F.~Deng, Z.~Wang, L.~Zhang, and X.~Zhou.
	\newblock New characterization of plurisubharmonic functions and positivity of
	direct image sheaves.
	\newblock {\em arXiv:1809.10371}.
	
		\bibitem{DWZZ2}
	F.~Deng, Z.~Wang, L.~Zhang, and X.~Zhou.
	\newblock Linear invariants of complex manifolds and their plurisubharmonic variations.
	\newblock {\em arXiv:1901.08920}.

	\bibitem{GZh12}
	Q.~Guan and X.~Zhou.
	\newblock Optimal constant problem in the {$L^2$} extension theorem.
	\newblock {\em C. R. Math. Acad. Sci. Paris}, 350(15-16):753--756, 2012.

	\bibitem{GZh15d}
	Q.~Guan and X.~Zhou.
	\newblock A solution of an {$L^2$} extension problem with an optimal estimate
	and applications.
	\newblock {\em Ann. of Math. (2)}, 181(3):1139--1208, 2015.
	\bibitem{HPS16}
	C.~Hacon, M.~Popa, and C.~Schnell.
	\newblock Algebraic fiber spaces over abelian varieties: {A}round a recent
	theorem by {C}ao and {P}\u aun.
	\newblock In {\em Local and global methods in algebraic geometry}, volume 712
	of {\em Contemp. Math.}, pages 143--195. Amer. Math. Soc., Providence, RI,
2018.

	\bibitem{Hor65}
	L.~H\"ormander.
	\newblock {$L^{2}$} estimates and existence theorems for the {$\bar \partial
	$}\ operator.
	\newblock {\em Acta Math.}, 113:89--152, 1965.

	\bibitem{HI}
	G.~Hosono and T.~Inayama.
	\newblock A converse of H\" ormander's $L^2$-estimate and new positivity
	notions for vector bundles.
	\newblock {\em arXiv:1901.02223v1.}
	
	
	\bibitem{LS14}
	L.~Lempert and R.~Sz\H oke. \emph{Direct images, fields of {H}ilbert
		spaces, and geometric
		quantization.}
	Comm. Math. Phys., 327(1):49--99, 2014.
	
	
	
	\bibitem{LiYa14}
	K.~Liu and X.~Yang.
	\newblock Curvatures of direct image sheaves of vector bundles and
	applications.
	\newblock {\em J. Differential Geom.}, 98(1):117--145, 2014.
	
	\bibitem{MoTa08}
	C.~Mourougane and S.~Takayama.
	\newblock Hodge metrics and the curvature of higher direct images.
	\newblock {\em Ann. Sci. \'{E}c. Norm. Sup\'{e}r. (4)}, 41(6):905--924, 2008.



	\bibitem{OT1}
	T.~Ohsawa and K.~Takegoshi.
	\newblock On the extension of {$L^2$} holomorphic functions.
	\newblock {\em Math. Z.}, 195(2):197--204, 1987.
	
	\bibitem{PT18}
	M.~P\u aun and S.~Takayama.
	\newblock Positivity of twisted relative pluricanonical bundles and their
	direct images.
	\newblock {\em J. Algebraic Geom.}, 27(2):211--272, 2018.
\bibitem{Raufi 13} H. Raufi, Log concavity for matrix-valued functions and a matrix-valued Pr¨¦kopa Theorem, Preprint available on ArXiv (2013).

	\bibitem{ZZ17}
X.~Zhou and L.~Zhu. \emph{An optimal $L^2$ extension theorem on
	weakly pseudoconvex k\"{a}hler manifolds.}
J. of Differential Geom., 110(1):135--186, 2018.
	
	\bibitem{ZhouZhu}
	X. Zhou and L. Zhu,
	\newblock Siu's lemma, optimal $L^2$ extension and applications to twisted pluricanonical sheaves.
	\newblock {\em Math.  Ann.}  https://doi.org/10.1007/s00208-018-1783-8.
	


	
	\bibitem{ZhouZhu192}
	X. Zhou and L. Zhu,
	\newblock Optimal $L^2$ extension of sections from subvarieties in weakly pseudoconvex manifolds.
	\newblock {\em arXiv: 1909.08820v1.}
	

\end{thebibliography}
\end{document}